\documentclass{amsart}

\usepackage{amsmath,amsthm,amscd,amssymb}
\usepackage{mathrsfs}
\usepackage{graphicx}
\usepackage{enumerate}
\usepackage{color}
\usepackage{here}

\usepackage{hyperref}
\usepackage{xcolor}
\hypersetup{
colorlinks=green,
citebordercolor=green,
linkbordercolor=green,
urlbordercolor=green
}

\newtheorem{theorem}{Theorem}[section]

\newtheorem{lemma}[theorem]{Lemma}
\newtheorem{proposition}[theorem]{Proposition}

\theoremstyle{definition}
\newtheorem{definition}[theorem]{Definition}
\newtheorem{remark}[theorem]{Remark}

\newtheorem{mtheorem}{Theorem}

\newtheorem{mcorollary}[mtheorem]{Corollary}
\newtheorem{question}[theorem]{Question}
\numberwithin{equation}{section}
\numberwithin{figure}{section}

\title[A robust obstruction to full strong pluripotency]{A robust obstruction to full strong pluripotency for wild blender-horseshoes}

\date{\today}

\author{Shin Kiriki}
\address[Shin Kiriki]{Department of Mathematics, Tokai 
University, 4-1-1 Kitakaname, Hiratsuka, Kanagawa, 259-1292, JAPAN}
\email{kiriki@tokai.ac.jp}

\author{Xiaolong Li}
\address[Xiaolong Li]{School of Mathematics and Statistics, 
Huazhong University of Science and Technology, Luoyu Road 1037, Wuhan, 430074, CHINA}
\email{lixl@hust.edu.cn}

\author{Yushi Nakano}
\address[Yushi Nakano]{
Department of Mathematics, Hokkaido University, Kita 10, Nishi 8, Kita-Ku, Sapporo, Hokkaido, 060-0810, 
JAPAN}
\email{yushi.nakano@math.sci.hokudai.ac.jp}

\author{Teruhiko Soma}
\address[Teruhiko Soma]{Department of Mathematical Sciences, 
Tokyo Metropolitan University, 1-1 Minami-Osawa, Hachioji, Tokyo, 192-0397, JAPAN}
\email{tsoma@tmu.ac.jp}

\subjclass[2020]{Primary: 37C20, 37C29; Secondary: 37C25}

\keywords{Takens' Last Problem, strong pluripotency, empirical measure, blender-horseshoe, homoclinic tangency}

\begin{document}

\begin{abstract}
Suppose that $M$ is a closed manifold of dimension greater than two and
$r \ge 2$.
We show that there exists a \(C^r\)-diffeomorphism
\(f:M\to M\) with a wild affine blender-horseshoe
\(\Lambda_f\) such that \(f\) is \(C^r\)-robustly strongly
pluripotent, in the sense that after arbitrarily small
perturbations a positive-Lebesgue-measure family of orbits can
shadow a prescribed orbit in time average, for the dense subset
\(\Lambda_f^{(\mathrm{mj})}\) consisting of itineraries satisfying
the majority condition, namely, those for which the symbol \(0\)
occupies at least half of certain growing terminal blocks, but not
for the whole horseshoe \(\Lambda_f\).
\end{abstract}

\maketitle

\section{Introduction}

Suppose that $M$ is a closed Riemannian manifold and $\mathrm{Diff}^r(M)$ $(r\geq 1)$ is the space of 
$C^r$-diffeomorphisms on $M$ with $C^r$-topology.
For any $f\in \mathrm{Diff}^r(M)$ and $x\in M$, consider the sequence of empirical  measures 
\begin{equation}\label{eqn_empirical}
\delta^{n}_{x,f}=
\frac{1}{n}\sum_{i=0}^{n-1}\delta_{f^i(x)},
\end{equation}
where $\delta_{f^i(x)}$ is the Dirac measure on $M$ supported at $f^i(x)$.
Let $\mathcal{P}(M)$ be the space of Borel probability measures on $M$ with weak$^*$-topology.
The notion of historic behavior is  introduced by Ruelle \cite{R01}.
We say that, for  $x\in M$, the forward orbit $(f^n(x))_{n\geq 0}$ has 
\emph{historic behavior} if the sequence of empirical measures \eqref{eqn_empirical} does not converge in $\mathcal{P}(M)$.
A subset of $M$ is called a \emph{historic set} or a \emph{set with historic behavior} if it consists of initial points which give rise to orbits with historic behavior.
Note that the maximal historic set is Lebesgue measurable, see \cite[Proposition A.1]{KLNS22}.

The following problem proposed by Takens \cite{T08} is the main motivation of 
our study, which is nowadays called Takens' Last Problem 
(abbreviated as TLP for short).

\medskip
\emph{Are there persistent classes of smooth dynamical systems such that the maximal historic 
set has positive Lebesgue measure?}
\medskip

Here a non-empty subset $\mathcal{C}$ of $\mathrm{Diff}^r(M)$ is called a \emph{class}.
A property $\mathbf{P}$ on elements of $\mathrm{Diff}^r(M)$ is $\mathcal{C}$-\emph{persistent} or \emph{persistent with respect to} 
$\mathcal{C}$ if every element of $\mathcal{C}$ satisfies the property.
Moreover, if the class $\mathcal{C}$ has non-empty interior, then $\mathbf{P}$ 
is called a $C^r$-\emph{robust} property of $\mathrm{Diff}^r(M)$.

An example of 2-dimensional dynamics with observable historic set is given by 
Colli-Vargas \cite{CV01}, where 
a subset of $M$ is called \emph{observable} if it has positive Lebesgue measure.
Their model is important in the sense that it provides a template for constructing 
diffeomorphisms with an observable historic set of special type, i.e.\ historic contracting wandering domains.
Here a non-empty connected open subset $D$ of $M$ is called a  \emph{wandering domain} 
if, for any $m,n\in \mathbb{N}\cup \{0\}$ $(m\neq n)$, 
$f^m(D)\cap f^n(D)=\emptyset$.
In fact, they present in \cite{KNS23} a class $\mathcal{C}$ of diffeomorphisms of dimension three or more with wild affine blender-horseshoes such that each element of $\mathcal{C}$ has a 
contracting wandering domain with historic behavior.
Recall that a blender is, in particular, a locally maximal uniformly hyperbolic set with an additional robust intersection property; the precise blender-horseshoe setting used in this paper will be described later. Also, for \(r\geq 2\), a dissipative Newhouse domain is a \(C^r\)-open set \(\mathcal N\subset \operatorname{Diff}^r(M)\) in which diffeomorphisms exhibiting a homoclinic tangency associated with a dissipative saddle are dense.
Note that, in the setting of diffeomorphisms, the blender mechanism requires ambient dimension at least three.

By incorporating the Colli-Vargas model, 
it is shown 
in \cite{KS17} that there exists a dense subset $\mathcal{C}$ of $\mathcal{N}$ 
each element of which admits a contracting wandering domain with historic behavior.
This is an affirmative answer to TLP relative to a dense subset of $\mathcal{N}$.
Subsequently, other answers to TLP relative to certain persistent classes 
are given by Labouriau-Rodrigues \cite{LR17}, Barrientos \cite{B22}, Berge-Biebler \cite{BB23}, 
Jamilov-Mukhamedov \cite{JM24} and so on.

Taking one more step beyond TLP, we study the detailed behavior of orbits from statistical and, more strongly, geometrical points of view. 
The notions of pluripotency and strong pluripotency used here are recalled in Definition \ref{d_pluri} below. 
For the particular two-dimensional Colli--Vargas model \(f_{0}\) considered in \cite{CV01}, the arguments there implicitly imply that the geometrical average of the orbit of any point in the whole horseshoe \(\Lambda_{f_{0}}\) can be realized observably by orbits of a diffeomorphism arbitrarily \(C^{r}\)-close to \(f_{0}\); in our terminology, \(f_{0}\) has full strong pluripotency. Moreover, it was shown in \cite{KLNSV} that every diffeomorphism \(f\) sufficiently close to this particular \(f_{0}\) has the same property for the continuation \(\Lambda_{f}\) of \(\Lambda_{f_{0}}\). Thus, full strong pluripotency is \(C^{r}\)-robust in a neighborhood of the Colli--Vargas model.

In contrast, for the higher-dimensional affine blender-horseshoe model considered in \cite{KNS}, strong pluripotency was established only for 
the subset $\Lambda_f^{(\mathrm{mj})}\subset\Lambda_f$,
where $(\mathrm{mj})$ stands for ``majority condition'';
see \eqref{eqn_mj_cond} in the next section for its definition.
Furthermore, this strong pluripotency on \(\Lambda_{f}^{(\mathrm{mj})}\) is \(C^{r}\)-robust. 
The aim of the present paper is to show that, within this affine blender-horseshoe family, this robust strong pluripotency cannot be extended to the whole horseshoe.
Thus, the comparison made here is between these two specific model families and should not be interpreted as a general dichotomy between two- and higher-dimensional dynamics.

Our approach is based on the Pluripotency Lemma in \cite{KNS}, which
characterizes strong pluripotency in terms of describability.
Roughly speaking, describability means that, after an arbitrarily
small perturbation, a positive-Lebesgue-measure family of orbits
follows a prescribed symbolic itinerary at a set of time indices
of asymptotic density one; see Section \ref{S_describable} for the precise definition.

In previous works, the Pluripotency Lemma was mainly used in the
forward direction: one first established describability and then
deduced strong pluripotency. 
In the present paper, we use the converse direction. Namely, if full strong pluripotency held, then every symbolic itinerary would have to be describable. 
We choose an asymptotically $1$-filling itinerary and show that its
describability would force unbounded growth of certain projected
areas, contradicting the bounded geometry of the construction.

\section{Pluripotency and main theorem}\label{S_mian_thm}

As in Introduction, suppose that $M$ is a closed Riemannian manifold and $r\geq 1$.
We consider the \emph{first Wasserstein metric} $d_W$ on the space $\mathcal{P}(M)$ of Borel probability measures on $M$ defined as
\[
d_W(\mu ,\nu )= \sup_\varphi \left\vert \int _M \varphi \, d\mu - \int _M \varphi \, d\nu \right\vert
\]
for $\mu$, $\nu\in \mathcal{P}(M)$, 
where the supremum is taken for all Lipschitz functions $\varphi : M \longrightarrow [-1, 1]$ 
with Lipschitz constant bounded by $1$. 
Note that the metric $d_W$ is compatible with the weak$^*$-topology on $\mathcal{P}(M)$.

Now we recall the concepts of pluripotency and strong pluripotency introduced in \cite{KNS}.

\begin{definition}[Pluripotency]\label{d_pluri}
Let $\Lambda_f $ be a uniformly hyperbolic compact invariant set for $f \in \mathrm{Diff}^{r}(M)$.
\begin{enumerate}[(1)] 
\item $f$  is \emph{pluripotent} for  a subset $\Lambda_f'$ of $\Lambda_f$ if, for any $x\in \Lambda_f'$,  
there exist $g \in \mathrm{Diff}^r(M)$ arbitrarily $C^r$-close to $f$ and 
a positive Lebesgue measure subset $D_g$ of $M$ such that, 
for any $y\in D_g$, 
\begin{equation}\label{defpl}
\lim _{n\to \infty} d_W( \delta_{y,g}^n, \delta_{x_g,g}^n ) =0,
\end{equation}
where $x_g\in \Lambda_g'$ is the continuation of $x\in \Lambda_f'$. 
\item
$f$ is \emph{strongly pluripotent} for $\Lambda_f'$ if the 
following condition holds instead of \eqref{defpl} for $g$ and $D_g$ as above.
\begin{equation}\label{defspl}
\lim _{n\to \infty} 
\frac{1}{n}\sum_{i=0}^{n-1}
\sup_{y\in D_{g}}\{\mathrm{dist}(g^i(y),g^i(x_{g}))\}=0.
\end{equation}
\end{enumerate}
Moreover such an $f$ is said to be \emph{fully} (strongly) pluripotent if $\Lambda_f'$ is equal to the whole 
invariant set $\Lambda_f$.
\end{definition}
We note that \eqref{defspl} implies \eqref{defpl}, while the converse is not true in 
general.
For example, see \cite[Theorem 1.8]{KNS}.

In this paper we consider the case of $\dim M\geq 3$.
In \cite[Theorem B]{KNS}, we showed that 
there exists a neighborhood $\mathcal{U}_0$ of a diffeomorphism $f_0$ in $\mathrm{Diff}^r(M)$ with a wild affine blender-horseshoe 
such that any element $f$ of $\mathcal{U}_0$ has the continuation $\Lambda_f$ of 
$\Lambda_{f_0}$ and is strongly pluripotent for 
the subset $\Lambda_f^{(\rm mj)}$ of 
$\Lambda_f$ consisting of elements with the majority condition.
Here we say that an element $x$ of $\Lambda_f$ satisfies the \emph{majority condition} if 
\begin{equation}\label{eqn_mj_cond}
\liminf_{n\to\infty}\frac{
\#\left\{
j\ ;\ 
n- (3n)^{2/3} \leq j\leq n,\ 
v_{j}=0
\right\}
}{(3n)^{2/3}}\geq \frac12,
\end{equation}
where 
$(v_j)_{j\in \mathbb{Z}}\in \{0,1\}^{\mathbb{Z}}$ is the binary code corresponding to $x$.
In other words, for large $n$, the condition \eqref{eqn_mj_cond} looks at the
terminal block of approximately $(3n)^{2/3}$ symbols ending at
time $n$ and requires that the proportion of $0$'s is asymptotically bounded from below by $1/2$.
In fact, they considered the case that the length $\alpha_{k+1}-\alpha_k$ of the $k$-th block is $k^2+O(k)$.
Since $\sum\limits_{i=1}^k i^2=k(k+1)(2k+1)/6$, $\alpha_{k+1}-\alpha_k=(3n)^{2/3}+O(n^{1/3})$ if $\alpha_{k+1}=n$.
We refer to the proof of \cite[Lemma~9.1]{KNS} for details.

Since $\Lambda_f^{({\rm mj})}$ is dense in but not equal to $\Lambda_f$, the following question  is natural.

\medskip
\emph{
Does there exist a diffeomorphism $f$ arbitrarily $C^r$-close to $f_0$ 
as  above and strongly pluripotent for the whole $\Lambda_{f}$?}
\medskip

The following theorem gives a negative answer to this question.

\begin{mtheorem}\label{thm_A}
Let $M$ be a closed manifold of dimension greater than two. 
Then, for every $1\leq r\leq \infty$,  there exist $f_0\in \mathrm{Diff}^r(M)$ 
with an affine blender-horseshoe $\Lambda_{f_0}$ and a $C^r$-neighborhood $\mathcal{U}_0$ of $f_0$ satisfying the following condition. 
Any element $f$ of  $\mathcal{U}_0$ is not strongly pluripotent 
for the continuation $\Lambda_f$ of the whole $\Lambda_{f_0}$.
\end{mtheorem}

The diffeomorphism $f_0$ given in Theorem \ref{thm_A} is essentially the same as the one in \cite[Theorem~B]{KNS}, with the extra assumptions in Section \ref{S_2} added.
Theorem \ref{thm_A} exhibits an obstruction to full strong pluripotency in the present family.

By Theorem \ref{thm_A} together with \cite[Theorem B]{KNS}, we have the following corollary.

\begin{mcorollary}\label{cor_B}
If $2\leq r<\infty$, then there exist $f_0$ and $\mathcal{U}_0$ as in Theorem \ref{thm_A} 
such that  
any element $f$ of $\mathcal{U}_0$ is strongly pluripotent for $\Lambda_f^{({\rm mj})}$ but not 
for $\Lambda_f$.
\end{mcorollary}

We note that, to apply Theorem B in \cite{KNS}, $f_0$ in Corollary \ref{cor_B} is 
a $C^r$-diffeomorphism with $r\geq 2$ which has a wild blender-horseshoe.
Here a blender-horseshoe $\Lambda$ is called \emph{wild} 
if $W^{\mathrm{u}}(\Lambda)$ and $W^{\mathrm{s}}(\Lambda)$ have 
a homoclinic tangency in the closure of the superposition region of $\Lambda$, 
see \cite[Section 6.2]{BDV05} for such a region.

Corollary \ref{cor_B} raises the following question.

\begin{question}
Let
$$
\Lambda_{f_0}^{(\mathrm{sp})}
:=
\{x\in\Lambda_{f_0}; \text{$f_0$ is strongly pluripotent with respect to $x$}\}.
$$
Can one characterize \(\Lambda_{f_0}^{(\mathrm{sp})}\)? In particular, how does it compare with the majority-condition subset \(\Lambda_{f_0}^{(\mathrm{mj})}\)?
\end{question}

Since $\Lambda_{f_0}^{(\mathrm{mj})}$ is dense in
$\Lambda_{f_0}$ and contained in $\Lambda_{f_0}^{\mathrm{sp}}$,
the set $\Lambda_{f_0}^{\mathrm{sp}}$ is dense in $\Lambda_{f_0}$.
As it is shown in Section \ref{S_proof_of_thmA}, the complement 
$\Lambda_{f_0} \setminus \Lambda_{f_0}^{\mathrm{sp}}$ is also dense in $\Lambda_{f_0}$.

Theorem \ref{thm_A} says nothing about the (weak) pluripotency for $\Lambda_f$. 
So the following question remains open.

\begin{question}
Does there exist a diffeomorphism $f$ which is arbitrarily close to $f_0$ given in Theorem \ref{thm_A} 
and pluripotent for the whole $\Lambda_f$?
\end{question}

We next relate the preceding discrete-time construction to continuous-time dynamics. Since Poincar\'e return maps provide a natural bridge between diffeomorphisms and flows, one may ask whether the robust dichotomy established in Corollary \ref{cor_B} can also be observed within a family of vector fields. 
The following result shows that this is indeed the case, in a form that is stable under perturbations of the vector field.

\begin{mcorollary}\label{cor_C}
Let \(2\leq r<\infty\). 
There exist a closed smooth manifold \(N\) of dimension greater than three, a \(C^r\) vector field \(X_0\) on \(N\), and a \(C^r\)-neighborhood \(\mathcal V_0\) of \(X_0\) such that every \(X\in\mathcal V_0\) admits a global cross-section whose first-return map exhibits the same pluripotency dichotomy as in Corollary \ref{cor_B}. 
Moreover, the dichotomy can be realized within arbitrarily small
$C^r$-perturbations of the vector field.
\end{mcorollary}

In the realization underlying Corollary~\ref{cor_C}, $N$ is the
mapping torus of the diffeomorphism $f_0$ in
Corollary~\ref{cor_B}, and $X_0$ is the corresponding suspension
vector field; the canonical copy of $M$ is a global cross-section
with first-return map $f_0$.
The precise formulation, including the realization of the
perturbing diffeomorphisms as Poincar\'e maps of nearby vector
fields, is given in Section~\ref{S_flow}.

\section{Describability and Pluripotency Lemma}\label{S_describable}

In this section, we present a practical necessary and sufficient condition for a diffeomorphism $f$ on 
$M$ to be strongly pluripotent.

A pair $\{\mathbb{U}_0,\mathbb{U}_1\}$ of connected open sets in $M$ with mutually disjoint closures 
 is called a  \emph{coding pair} of a horseshoe $\Lambda_f$ of $f$ if 
\[
\Lambda_f =\bigcap_{i\in \mathbb{Z}} f^{i}(\mathbb U_0 \sqcup \mathbb U_1)
\]
and the restriction $f|_{\Lambda_f}:\Lambda_f\longrightarrow \Lambda_f$ is topologically conjugate 
to the shift map on $ \{ 0,1\} ^{\mathbb Z}$ by the coding map  
$\mathcal{I}_f: \Lambda_f  \longrightarrow  \{ 0,1\} ^{\mathbb Z}$ defined as 
\[
\left(\mathcal{I}_f(x) \right)_j = v \quad \text{if $f^j(x) \in \mathbb U_{v}$}, 
\]
where $(\mathcal{I}_f(x) )_j$ stands for the $j$-th entry of $\mathcal{I}_f(x)$.

\begin{definition}[Describability]\label{describable}
Let $\Sigma$ be a subset of $\{0,1\}^{\mathbb N_0}$ and 
$f$ an element of $\mathrm{Diff}^{r}(M)$ with a horseshoe $\Lambda_f$ associated with a coding pair $\{ \mathbb U_0, \mathbb U_1 \}$, 
where $\mathbb N_0=\mathbb{N}\cup \{0\}$.  
We say that $f$ is $\Sigma$-\emph{describable} over $\Lambda_f$  
if any element $\underline{v}=(v_{0}v_{1}v_{2}\ldots\,)$ of $\Sigma$
satisfies the following conditions: 

\begin{enumerate}[({D}1)]
\makeatletter
\renewcommand{\p@enumi}{D}
\makeatother
\item\label{D1}
There exist a strictly increasing sequence $(\alpha _k)_{k\in \mathbb N}$ of non-negative integers and 
integers $\beta _k$ $(k\in \mathbb N)$ with $0\leq \beta _k\leq \alpha _{k+1}-\alpha _k-1$ and  
such that 
\[
\lim _{n\rightarrow \infty} \frac{\#\left\{j\,;\, 0\le j \le n-1,\ j \in \bigcup _{k=1}^\infty \mathbb{I}_k\right\}}{n} =1, 
\]
where $\mathbb{I}_k
=[\alpha_k, \alpha _k+ \beta_k] \cap \mathbb Z$.
\item\label{D2}
There exist an element $g$ of $\mathrm{Diff}^{r}(M)$ arbitrarily $C^r$-close to $f$ and  
a positive Lebesgue measure subset $D_g$ of $M$  
such that  
\[
g^j (D_g) \subset \mathbb U_{v_{j}}
\]
for any  $j\in \bigcup_{k\in \mathbb{N}
} \mathbb I_k$.  
See Figure \ref{f_describability}.
\end{enumerate}
\end{definition}
\begin{figure}[hbtp]
\centering
\scalebox{0.6}{\includegraphics[clip]{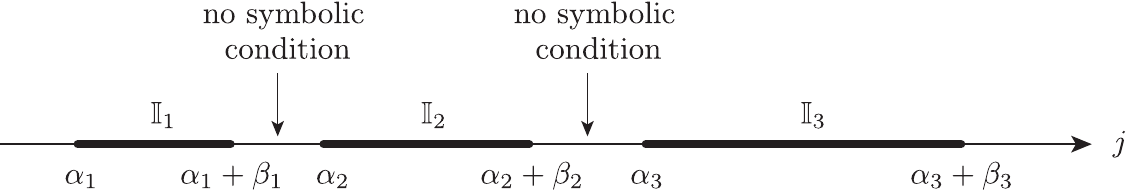}}
\caption{The union $J=\bigcup_{k\geq 1}\mathbb{I}_k$ has asymptotic density one, 
that is, $\dfrac{\#(J\cap \{0,\dots,n-1\})}{n}\to 1$.
For any $j\in \mathbb{I}_k$, $g^j(D_g)$ is contained in $\mathbb{U}_{v_j}$.
No condition is imposed on the intervening gaps.}
\label{f_describability}
\end{figure}

The following theorem  in \cite[Theorem A]{KNS} plays 
an important role in the proof of the main theorem, 
where $\widehat \Sigma$ is the subset of $\{0,1\}^{\mathbb{Z}}$ consisting of elements 
$(v_j)_{j\in \mathbb{Z}}$ with $(v_j)_{j\in \mathbb{N}_0}\in \Sigma$.

\begin{theorem}[Pluripotency Lemma]\label{p-lemma}
Suppose that $f$ is an element of $\mathrm{Diff}^r(M)$ $(r\geq 1)$ with a horseshoe $\Lambda_f$ associated with a coding pair $\{ \mathbb U_0, \mathbb U_1 \}$ and 
$\Sigma$ is a non-empty subset of $\{0,1\}^{\mathbb N_0}$.
Then $f$ is $\Sigma$-describable if and only if 
$f$ is strongly pluripotent for $\mathcal{I}_f^{-1}(\widehat \Sigma)$.
\end{theorem}

In Section \ref{S_proof_of_thmA}, we will show that, for asymptotically 1-filling itineraries (see \eqref{eqn_1-filling} for the definition), the describability would force long good blocks whose projected-area growth is incompatible with 
the geometric phenomena caused by the $g$-orbit of $D_g$. 
This obstruction rules out full strong pluripotency.

\section{Wild affine blender-horseshoes}\label{S_2}

For simplicity, we first describe the model in dimension three. The extension to \(\dim M>3\) is obtained by adding strongly stable coordinates, while keeping the unstable and center-stable directions one-dimensional. Thus, in higher dimensions, the relevant splitting has the form
$$
E^{\mathrm u}\oplus E^{\mathrm{ss}}\oplus E^{\mathrm{cs}},
\quad
\dim E^{\mathrm u}=\dim E^{\mathrm{cs}}=1,
$$
with all additional directions contained in \(E^{\mathrm{ss}}\).

The area-growth argument used later involves only the two-dimensional \(E^{\mathrm u}\oplus E^{\mathrm{cs}}\) directions. 
We do not consider higher-dimensional volumes involving strongly stable directions, since the contraction in those directions would destroy the expansion mechanism needed in the proof. 
When a positive-Lebesgue-measure set is considered later, Fubini's theorem will be used to select a two-dimensional slice obtained by fixing the strongly stable coordinates and having positive two-dimensional Lebesgue measure. 
The subsequent projected-area argument is then applied to this slice. 
See Section \ref{S_proof_of_thmA} for the precise implementation.

\medskip

A diffeomorphism $f_0$ on $M$ presented in this section has a wild affine blender-horseshoe similar to those in \cite{KNS23, KNS}. 
Let $\lambda_{\rm ss}, \lambda_{\rm cs0}, \lambda_{\rm cs1}$ and $\lambda_{\rm u}$ 
be real positive constants with
\begin{equation}\label{eqn_eigen_v}
\lambda_{\rm ss}<\frac12,\quad
\lambda_{\rm ss}<\lambda_{\rm cs0}\leq \lambda_{\rm cs1}<1<\lambda_{\rm cs0}+\lambda_{\rm cs1},\quad 2<\lambda_{\rm u}.
\end{equation}
Here the inequalities $\lambda_{\rm ss}<2^{-1}$ and $2<\lambda_{\rm u}$ are necessary to define a blender-horseshoe.
The condition $1<\lambda_{{\rm cs}0}+\lambda_{{\rm cs}1}$ is used to show that the blender-horseshoe has the \emph{distinctive property}.  
See Section 6.2 in \cite{BDV05} for details.
The property is required in the proof of \cite[Theorem B]{KNS}, 
which is essential for proving Corollary \ref{cor_B}.
It follows from \eqref{eqn_eigen_v} that   
\begin{equation}\label{eqn_lamcs1}
\lambda_{\rm cs1}\lambda_{\rm u}>1,
\end{equation}
which is crucial in the proof of Theorem \ref{thm_A}.
We fix a sufficiently small  $\varepsilon_0>0$.
For example, it is taken so that $1/2-\lambda_{\rm ss}(1+\varepsilon_0)$ is positive.
The closed interval $[-\varepsilon_{0},1+\varepsilon_{0}]$ is denoted by  $I_{\varepsilon_{0}}$ 
and the cube $I_{\varepsilon_{0}}^{3}=I_{\varepsilon_{0}}\times I_{\varepsilon_{0}}\times I_{\varepsilon_{0}}$ by $\mathbb{B}$.
We suppose that $\mathbb{B}$ is embedded in $M$ 
and $M$ has a Riemannian metric extending the standard Euclidean metric on $\mathbb{B}$.
Consider the sub-blocks of $\mathbb{B}$ defined as 
\[
\mathbb{V}_{0}= [-\varepsilon_{0}, \lambda_{\rm u}^{-1}+\varepsilon_0]\times I_{\varepsilon_{0}}^{2},\quad 
\mathbb{V}_{1}= [1-\lambda_{\rm u}^{-1}-\varepsilon_0,1+\varepsilon_{0}]\times I_{\varepsilon_{0}}^{2}. 
\]
Let $f_{0}$ be a $C^r$-diffeomorphism on $M$ such that $f_{0}|_{\mathbb{V}_{0}\cup \mathbb{V}_{1}}$ is defined as
\begin{equation}\label{eqn_f0V}
f_{0}(x,y,z)=
\begin{cases}
(\lambda_{\mathrm{u}}x,\lambda_{\mathrm{ss}}y,\zeta_0(z))&\text{if}\ (x,y,z)\in \mathbb{V}_{0},\\
(\lambda_{\mathrm{u}}(1-x),1-\lambda_{\mathrm{ss}}y,\zeta_1(z)) &\text{if}\ (x,y,z)\in \mathbb{V}_{1},
\end{cases}
\end{equation} 
where 
$\zeta_0$ and $\zeta_1$ are the affine maps on $I_{\varepsilon_0}$ given by
\[
\zeta_0(z)=\lambda_{\mathrm{cs} 0}z\quad\text{and}\quad \zeta_1(z)=\lambda_{\mathrm{cs} 1}z+1-\lambda_{\mathrm{cs} 1}.
\]
From our setting, $f_{0}$ has the uniformly hyperbolic set  
\[
\Lambda_{f_{0}}=\bigcap_{n\in \mathbb{Z}} f_{0}^n(\mathbb{V}_{0}\cup \mathbb{V}_{1}),
\]
which belongs to the class of affine blender-horseshoes, see \cite[Section 1]{BD96} and \cite[Section 3]{BD12} for details. 
Note that \(f_0\) has the saddle fixed point
$$
p_{f_0}=(0,0,0)\in V_0.
$$
By the linear form of \(f_0|_{V_0}\), its local invariant manifolds of $p_{f_0}$ in $\mathbb{B}$ are tangent to the coordinate subspaces given by
\begin{align*}
W_{\mathrm{loc}}^{\mathrm{u}}(p_{f_0},f_0)
 &=\{(x,0,0);\,-\varepsilon_0\leq x\leq 1+\varepsilon_0\},\\
 W_{\mathrm{loc}}^{\mathrm{s}}(p_{f_0},f_0)
 &=\{(0,y,z);\,-\varepsilon_0\leq y,z\leq 1+\varepsilon_0\}.
\end{align*}

Let $a_{1},a_{2},a_{3},a_{4}$ and $\mu$ be constants 
satisfying  
\begin{subequations}
\begin{align}
&a_{1}>0,\quad a_2>0,\quad a_3a_4<0,\label{eqn_aaa}\\
&|a_{3}|\left(\frac12+\varepsilon_0\right)<\frac12-\lambda_{\rm ss}(1+\varepsilon_0),\label{eqn_|a3|}\\
&|a_3|<\min\left\{\dfrac{a_2|a_4|}{\sqrt{4a_1^2+a_2^2}},\dfrac{a_1a_2}{\sqrt{4a_1^2+a_2^2}}(1-2\varepsilon_0)\right\},\label{eqn_min}\\
&a_1^{-1}a_2(1+\varepsilon_0+\mu)^2<\left(\frac12-\lambda_{\rm u}^{-1}\right)^2,\  
a_1^{-1}a_2a_4^2(1+\varepsilon_0+\mu)<\left(\frac12+\varepsilon_0\right)^2,\label{eqn_a42}\\
&0<-\varepsilon_0+\mu,\quad a_2(1+\varepsilon_0+\mu)<\lambda_{\rm u}^{-1}.\label{eqn_e_0mu}
\end{align}
\end{subequations}
It is not hard to find such constants by taking $a_1$ sufficiently large and next $|a_3|$ sufficiently small.
The conditions other than \eqref{eqn_min} are used in this section, but 
\eqref{eqn_min} used in the proof of Lemma \ref{l_ea_horizontal}.

Now we consider the subsets 
\begin{align*}
\mathbb{X}_{0}&= [0, \lambda_{\rm u}^{-1}]\times I_{\varepsilon_{0}}^{2},\quad 
\mathbb{X}_{1}= [1-\lambda_{\rm u}^{-1},1]\times I_{\varepsilon_{0}}^{2},\\
\mathbb{X}_2&=\left\{(x,y,z)\in \mathbb{B}\,\Big|\,z\geq \frac{a_1}{a_2}\left(x-\frac12\right)^2-\mu\right\}
\end{align*}
of $\mathbb{B}$.
Note that $\mathbb{X}_{0}\subset \mathbb{V}_{0}$ and $\mathbb{X}_{1}\subset \mathbb{V}_{1}$.

We suppose that the $y$-axis direction in $\mathbb{R}^3$ is vertical and 
say that a surface in $\mathbb{R}^3$ that is parallel to the $xz$-plane is \emph{horizontal}.
Let $q_+$, $q_-$ be the intersection points of the parabolic cylinder 
\[
z=\dfrac{a_1}{a_2}\left(x-\dfrac12\right)^2-\mu,
\]
the vertical plane $z=1+\varepsilon_0$ and 
the horizontal plane $y=\frac12$ in $\mathbb{R}^3$.
Since the $x$-entry of $q_\pm$ is $\dfrac12\pm\sqrt{a_1^{-1}a_2(1+\varepsilon_0+\mu)}$, 
the former inequality of \eqref{eqn_a42} implies $\mathbb{X}_2\cap (\mathbb{X}_0\cup \mathbb{X}_1)=\emptyset$.
We denote by $\mathbb{Y}$ and $\mathbb{S}$ the closures of $[\lambda_{\rm u}^{-1}, 1-\lambda_{\rm u}^{-1}]\times I_{\varepsilon_0}^2\setminus \mathbb{X}_2$ and $\mathbb{B}\setminus [0,1]\times I_{\varepsilon_0}^2$ in $\mathbb{B}$ respectively.
The former inequality of \eqref{eqn_e_0mu} implies that $\mathbb{Y}$ consists of two components 
and hence the frontier of $\mathbb{X}_2$ in $\mathbb{B}$ consists of the two curved rectangles $F_2^{\pm}$ 
as illustrated in Figure \ref{f_X012}.
Here, for a subset $A$ of a topological space $B$,  the \emph{frontier} (or \emph{topological boundary}) of $A$ in $B$ is the intersection 
$\mathrm{Cl}(A)\cap \mathrm{Cl}(B\setminus A)$, which is denoted by $\mathrm{Fr}_B(A)$.
\begin{figure}[hbtp]
\centering
\scalebox{0.6}{\includegraphics[clip]{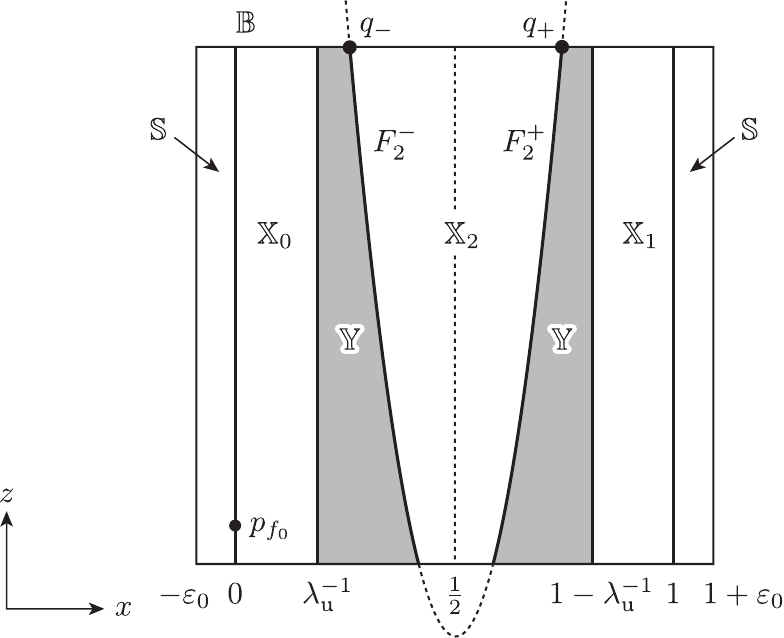}}
\caption{View from the bottom. $\mathrm{Fr}(\mathbb{X}_2)=F_2^-\cup F_2^+$.}
\label{f_X012}
\end{figure}
Simply we set $\mathrm{Fr}_{\mathbb{B}}(A)=\mathrm{Fr}(A)$ for any subset $A$ of $\mathbb{B}$.
Then 
\begin{align*}
\mathrm{Fr}(\mathbb{X}_0)&=\{0,\lambda_{\rm u}^{-1}\}\times I_{\varepsilon_0}^2, \quad 
\mathrm{Fr}(\mathbb{X}_1)=\{1-\lambda_{\rm u}^{-1},1\}\times I_{\varepsilon_0}^2
\quad\text{and}\\
\mathrm{Fr}(\mathbb{S})&=\{0,1\}\times I_{\varepsilon_0}^2.
\end{align*}

For any $(x,y,z)\in \mathbb{X}_{2}$ and constants $a_1$, $a_2$, $a_3$, $a_4$, $\mu$ satisfying  \eqref{eqn_aaa}--\eqref{eqn_e_0mu}, 
we suppose that   
\begin{equation}\label{eqn_tangency}
f_{0}^{2}(x,y,z)=\left(-a_{1}\Bigl(x-\frac{1}{2}\Bigr)^{2}+a_{2}(z+\mu),\ a_{3}\Bigl(y-\frac{1}{2}\Bigr)+\frac{1}{2},\ 
a_{4}\Bigl(x-\frac{1}{2}\Bigr)+\frac{1}{2} \right).
\end{equation}

The latter inequality of \eqref{eqn_a42} implies that $f_0^2(q_\pm)$ are contained in 
$\{0\}\times \mathrm{Int} I_{\varepsilon_0}^2$.
By this fact together with \eqref{eqn_e_0mu}, one can prove that $f_0^2(\mathbb{X}_2)$ is 
contained in $\mathbb{X}_0$.
See Figures \ref{f_blender} and \ref{f_Yf0} below.
The image of $f_0^2(\mathbb{X}_2)$ by the orthogonal projection to the $y$-axis is the closed interval 
$\left[-|a_3|\left(\frac12+\varepsilon_0\right)+\frac12,\,  |a_3|\left(\frac12+\varepsilon_0\right)+\frac12\right]$.
Hence, by \eqref{eqn_|a3|}, $f_0^2(\mathbb{X}_2)$ lies above $f_0(\mathbb{X}_0)$ and below $f_0(\mathbb{X}_1)$.
See Figure \ref{f_blender}. 
\begin{figure}[hbtp]
\centering
\scalebox{0.6}{\includegraphics[clip]{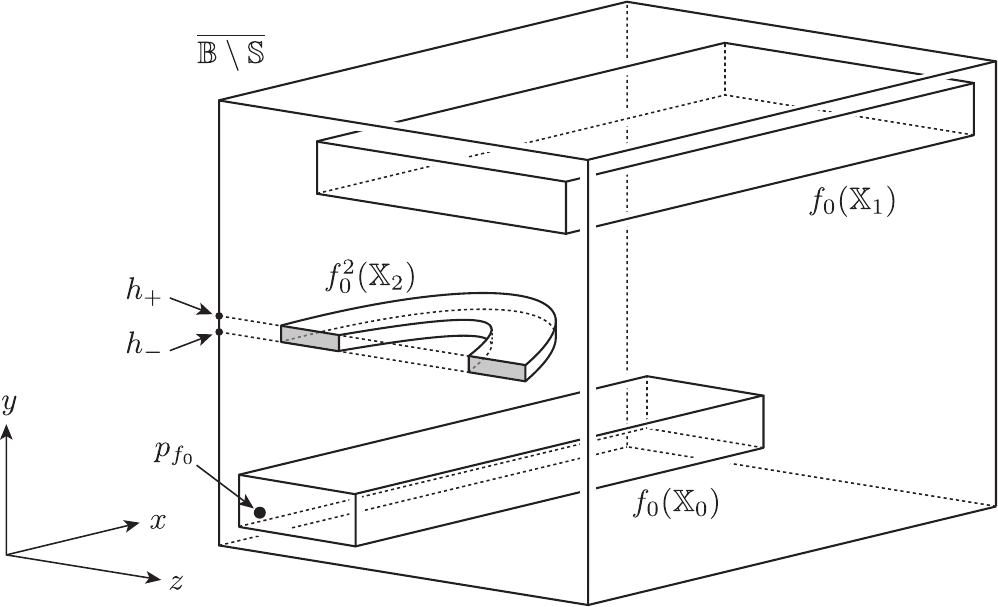}}
\caption{$h_\pm=\pm |a_3|\left(\frac12+\varepsilon_0\right)+\frac12$. 
The union of shaded rectangles represents $f_0^2(\mathrm{Fr}(\mathbb{X}_2))$.
The front square $\{0\}\times I_{\varepsilon_0}^2$ of $\overline{\mathbb{B}\setminus \mathbb{S}}$  is the component of $W^{\mathrm{s}}(p_{f_0},f_0)\cap \mathbb{B}$ containing $p_{f_0}$.}
\label{f_blender}
\end{figure}
From our definition \eqref{eqn_f0V} of $f_0$ on $\mathbb{V}_0\cup \mathbb{V}_1$, we know that the $f_0^2$-image of the frontier $\mathrm{Fr}(\mathbb{X}_0\cup \mathbb{X}_1)$ is also contained in 
$\{0\}\times \mathrm{Int} I_{\varepsilon_0}^2$.
As in \cite[Lemma 1.4]{KNS23}, one can choose $\mu$ in \eqref{eqn_tangency} 
so that $f_0$ has a $C^{1}$-robust homoclinic tangency associated with $\Lambda_{f_0}$. 
Then $\Lambda_{f_0}$ is called a \emph{wild affine blender-horseshoe}.

By \eqref{eqn_tangency}, for $\boldsymbol{x}=(x,y,z)\in \mathbb{X}_2$, 
\begin{equation}\label{eqn_DfX2}
Df_0^2(\boldsymbol{x})=\begin{pmatrix}-2a_1\left(x-\frac12\right)&0&a_2\\
0&a_3&0\\
a_4&0&0\end{pmatrix}.
\end{equation}
By \eqref{eqn_aaa}, $\det Df_0^2(\boldsymbol{x})=-a_2a_3a_4>0$.

By \eqref{eqn_f0V}, for any horizontal surface $H$ in $\mathbb{X}_i$ $(i=0,1)$ and a 
Lebesgue measurable subset $A$ of $H$, $f_0(H)$ is a horizontal surface in $\mathbb{B}$ and 
\begin{equation}\label{eqn_Leb_P0}
\mathrm{Leb}_{f_0(H)}(f_0(A))=\lambda_{{\rm cs}i}\lambda_{\rm u}\mathrm{Leb}_H(A),
\end{equation}
where $\mathrm{Leb}_H(\cdot)$ denotes the standard 2-dimensional Lebesgue measure on $H$.
On the other hand, by \eqref{eqn_tangency} and \eqref{eqn_DfX2}, for any horizontal surface $H'$ in $\mathbb{X}_2$ and a Lebesgue measurable subset $A'$ of $H'$, 
$f_0^2(H')$ is a horizontal surface in $\mathbb{X}_0$ and 
\begin{equation}\label{eqn_Leb_P}
\mathrm{Leb}_{f_0^2(H')}(f_0^2(A'))=|a_2a_4|\mathrm{Leb}_{H'}(A').
\end{equation}

From now on, we work under the following extra assumptions.

\subsection*{Extra Assumptions}
The intersection $f_0(\mathbb{X}_2)\cap \mathbb{B}$ is empty.
There exists an arc $\alpha$ in $M$ with $\mathrm{Int}\,\alpha\subset W^{\rm u}(p_{f_0},f_0)$ 
and connecting $p_{f_0}$ with an attracting fixed point  
$s_0$ of $f_0$ outside $\mathbb{B}$.
The sequence 
$f_0^n(\mathbb{Y}\cup \mathbb{S})$ converges to $\alpha$ as $n\to \infty$ with respect to the Hausdorff metric.
See Figure \ref{f_Yf0}. 
\begin{figure}[hbtp]
\centering
\scalebox{0.6}{\includegraphics[clip]{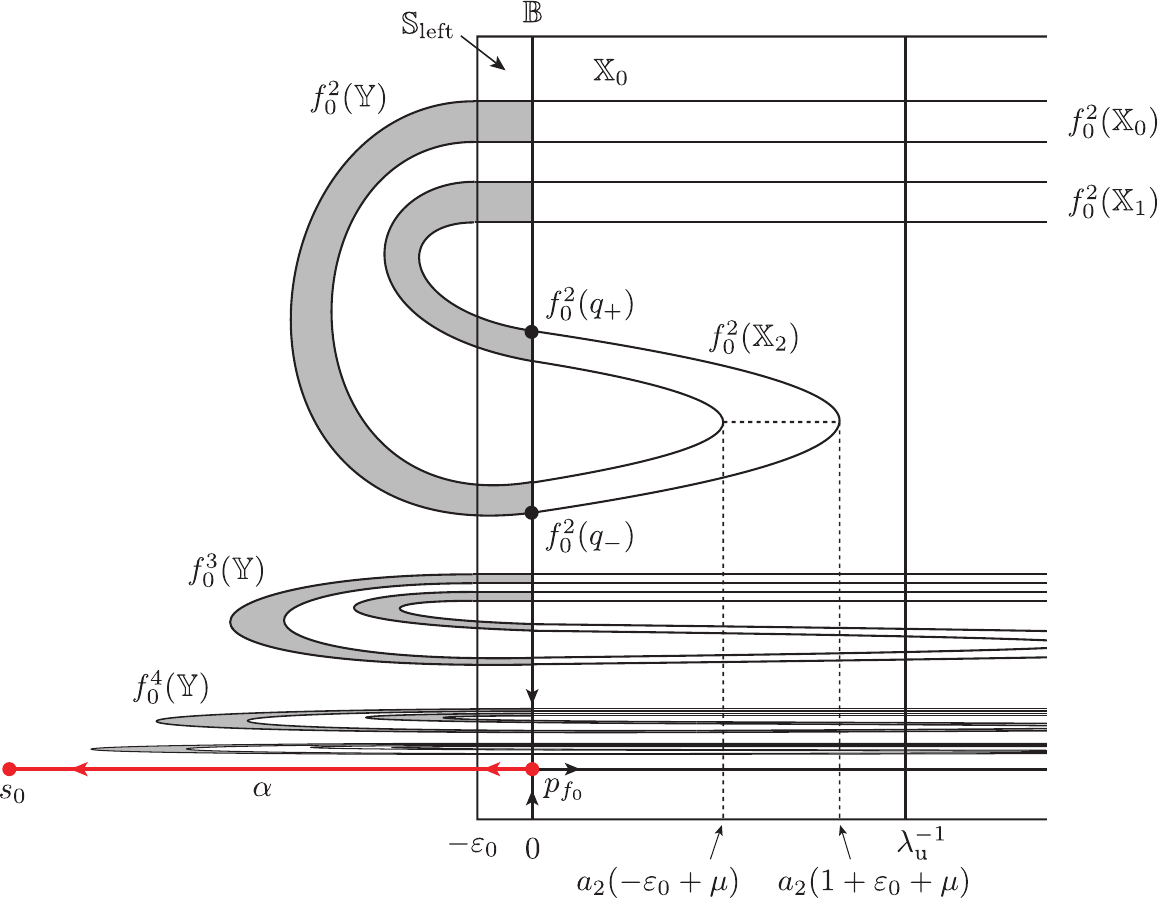}}
\caption{Schematic figure just presenting the connections of $f_0^2(\mathbb{X}_i)$ $(i=0,1,2)$ and $f_0^2(\mathbb{Y})$. 
They do not indicate their exact placements in $M$.
}
\label{f_Yf0}
\end{figure}

\medskip

The Hausdorff convergence in the last assumption is not meant to follow from the \(\lambda\)-lemma.
Let us briefly indicate how the assumption can be realized. 
Recall that $p_{f_0}=(0,0,0)$.
One may choose the extension of \(f_0\) outside \(\mathbb B\) so that \(M\) contains a box
$\mathbb B'=[-1-\varepsilon_0,0]\times I_{\varepsilon_0}^2$
adjacent to \(\mathbb B\).
We may take  
$$
f_0(x,y,z)
=
\bigl(h(x),\lambda_{\mathrm{ss}}y,\lambda_{\mathrm{cs}0}z\bigr)
$$
on \(\mathbb B'\), 
where \(h(-1)=-1\), \(h(0)=0\), $h'(x)>0$ on $[-1-\varepsilon_0,0]$, 
\(h'(x)<1\) on $[-1-\varepsilon_0,-1+\varepsilon_0]$, $h(x)<x$ on $(-1,0)$, and \(h(x)=\lambda_{\mathrm u}x\) on $[-\varepsilon_0,0]$. 
Then
$$
s_0=(-1,0,0)
\quad\text{and}\quad
\alpha=[-1,0]\times\{(0,0)\}
$$
have the properties stated above. 
Indeed, the \(y\)- and \(z\)-directions are uniformly contracted, while the images of an interval \([a,0]\) in the \(x\)-direction converge in the Hausdorff metric to \([-1,0]\). 
So the required Hausdorff convergence follows once the relevant forward image of $\mathbb Y\cup f(\mathbb{S})$ is arranged to enter this box.
This property is defined by strict trapping and contraction conditions and hence persist, after shrinking \(\mathcal U_0\), for every \(g\in\mathcal U_0\).

\section{Diffeomorphisms close to $f_0$}\label{S_near_f0}

\subsection{Extensions of $f_0|_{\mathbb{X}_0}$, $f_0|_{\mathbb{X}_1}$, 
$f_0^2|_{\mathbb{X}_2}$ to diffeomorphisms on $\mathbb{R}^3$}

Recall that
$
W_{\mathrm{loc}}^{s}(p_{f_0},f_0)
=
\{0\}\times I_{\varepsilon_0}^{2}.
$
By the invariance of the stable manifold, together with \eqref{eqn_f0V} and \eqref{eqn_tangency}, we have
$$
W^{s}(p_{f_0},f_0)
\supset
f_0^{-2}\!\left(
W_{\mathrm{loc}}^{s}(p_{f_0},f_0)
\right)
\supset
\operatorname{Fr}(\mathbb X_0)
\cup
\operatorname{Fr}(\mathbb X_1)
\cup
\operatorname{Fr}(\mathbb X_2).
$$
In particular, \(\operatorname{Fr}(\mathbb X_i)\) 
$(i=0,1,2)$ is a union of two components of $W^{\mathrm{s}}(p_{f_0},f_0)\cap \mathbb{B}$.

Consider a small neighborhood $\mathcal{U}_0$ of $f_0$ in $\mathrm{Diff}^r(M)$.
For any $g\in \mathcal{U}_0$, we denote by $p_g$ the saddle fixed point of $g$ obtained as the continuation of $p_{f_0}$.
For $i=0,1,2$, let $\mathbb{X}_{g,i}$  be the compact subset of $\mathbb{B}$ such that $\mathrm{Fr}(\mathbb{X}_{g,i})\subset W^{\mathrm{s}}(p_g,g)\cap \mathbb{B}$ is the continuation of $\mathrm{Fr}(\mathbb{X}_i)$.
The compact subset $\mathbb{Y}_g$ of $\mathbb{B}$ corresponding to $\mathbb{Y}$ 
is defined similarly.
We denote by $\mathbb{S}_g$ the closure of the union of components of $\mathbb{B}\setminus (\mathbb{X}_{g,0}\cup \mathbb{X}_{g,1})$ disjoint from $\mathbb{X}_{g,2}$.
One can suppose that $\mathbb{Y}_g\cup \mathbb{S}_g$ still satisfies Extra Assumptions with respect to $g$ instead of $f_0$ 
if necessary by replacing $\mathcal{U}_0$ with a smaller neighborhood of $f_0$.

For $i=0,1$, we consider the diffeomorphism $\varphi_i:\mathbb{R}^3\longrightarrow \mathbb{R}^3$ which has the same form as $f_0|_{\mathbb{V}_i}$ in \eqref{eqn_f0V} for \emph{any} $(x,y,z)\in \mathbb{R}^3$.
In particular, $\varphi_i|_{\mathbb{X}_i}=f_0|_{\mathbb{X}_i}$ holds.
Note that $\varphi_0$ is a linear map on $\mathbb{R}^3$ and $\varphi_1$ 
is an affine map fixing 
\[
(\lambda_{\mathrm u}(1+\lambda_{\mathrm u})^{-1},
(1+\lambda_{\mathrm{ss}})^{-1},1),
\] 
each of which preserves horizontal surfaces in $\mathbb{R}^3$.

Let $\eta:\mathbb{R}\longrightarrow \mathbb{R}$  be the unimodal $C^1$-function defined as 
\[
\eta(x)=\begin{cases}
2a_1\left(\dfrac12+\varepsilon_0\right)x+a_1\left(\dfrac12+\varepsilon_0\right)\!\left(-\dfrac12+\varepsilon_0\right) &
\text{if $x<-\varepsilon_0$},\\[4pt]
-a_1\left(x-\dfrac12\right)^2&\text{if $-\varepsilon_0\leq x\leq 1+\varepsilon_0$},\\[4pt]
-2a_1\left(\dfrac12+\varepsilon_0\right)x+a_1\left(\dfrac12+\varepsilon_0\right)\!\left(\dfrac32+\varepsilon_0\right) &
\text{if $x>1+\varepsilon_0$}.
\end{cases}
\]
From the definition, we have
\begin{equation}\label{eqn_d_eta}
\left|\frac{d\eta}{dx}(x)\right|\leq 2a_1\left(\frac12+\varepsilon_0\right)
\end{equation}
for  any $x\in\mathbb{R}$.
By using the function $\eta$, we define the $C^1$-diffeomorphism $\varphi_2:\mathbb{R}^3\longrightarrow \mathbb{R}^3$ by    
\begin{equation}\label{eqn_phi_2}
\varphi_{2}(x,y,z)=\left(\eta(x)+a_{2}(z+\mu),\ a_{3}\Bigl(y-\frac{1}{2}\Bigr)+\frac{1}{2},\ 
a_{4}\Bigl(x-\frac{1}{2}\Bigr)+\frac{1}{2} \right).
\end{equation}
By \eqref{eqn_d_eta}, the \((1,1)\)-entry of \(D\varphi_2(\boldsymbol{x})\) is uniformly bounded for all \(\boldsymbol{x}\in\mathbb R^3\). Notice that the corresponding entry
$
-2a_1\left(x-\frac12\right)
$
in \eqref{eqn_DfX2} is, of course, also uniformly bounded when restricted to the compact set \(\mathbb X_2\). 
The point is that, if the quadratic expression in \eqref{eqn_tangency} were extended unchanged to all of \(\mathbb R^3\), this entry would become unbounded as \(|x|\to\infty\). 
This is why, in defining the global extension \(\varphi_2\), we replace the quadratic term outside the relevant compact region by the affine pieces in the definition of \(\eta\). Thus, \(\varphi_2|_{\mathbb X_2}=f_0^2|_{\mathbb X_2}\) by \eqref{eqn_tangency}, while \(D\varphi_2\) has the uniform bound needed below.

\subsection{Almost horizontal surfaces in $\mathbb{R}^3$}\label{ss_almost_h}

To prove Theorem \ref{thm_A}, it is enough to show that, for the itinerary chosen below, no positive-Lebesgue-measure set \(D_g\) can satisfy condition \eqref{D2}.
The contradiction will be obtained from an area-expansion argument. 
More precisely, the full-dimensional volume is not the appropriate quantity to consider, since the strongly stable direction is contracting and hence the full Jacobian may be smaller than one. 
Instead, we restrict attention to the two-dimensional \(E^{\mathrm u}\oplus E^{\mathrm{cs}}\) directions, where the relevant projected Jacobian is expanding.

Starting from \(D_g\), we use Fubini's theorem to choose a two-dimensional slice \(A\) of positive area in the \(E^{\mathrm u}\oplus E^{\mathrm{cs}}\) directions. 
We then consider suitable components of selected iterates of \(A\) and estimate the areas of their projections onto the \(E^{\mathrm u}\oplus E^{\mathrm{cs}}\) coordinates. Condition \eqref{D2}, together with the prescribed itinerary, forces these selected components to remain in the bounded region under consideration, whereas the projected-area estimates obtained below show that their projected areas become arbitrarily large. 
This gives the desired contradiction.

Let $\varepsilon$ be a positive number sufficiently smaller than $\varepsilon_0$, which will be fixed in the proof of Lemma \ref{l_ea_horizontal}.
We consider the $\mathrm{uc}$-cone-field 
\[
\boldsymbol{C}_{\varepsilon}^{\mathrm{uc}}(\boldsymbol{x})=
\left\{
\boldsymbol{v}=(v^{\mathrm{u}},v^{\mathrm{s}},v^{\mathrm{cs}})\in T_{\boldsymbol{x}}(\mathbb{R}^3)\,;\,
|v^{\mathrm{s}}|\leq\varepsilon\sqrt{(v^{\mathrm{u}})^2+(v^{\mathrm{cs}})^2}
\right\}
\]
on $\mathbb{R}^3$.

We denote by $\mathrm{proj}:\mathbb{R}^3\longrightarrow \mathbb{R}^2$ the orthogonal projection to the $xz$-plane, that is, $\mathrm{proj}(x,y,z)=(x,z)$.
For any subset $K$ of $\mathbb{R}^3$, we set $\mathrm{proj}(K)=\widehat K$.
A surface $S$ in $\mathbb{R}^3$ is said to be $\varepsilon$-\emph{almost horizontal} if, for any 
$\boldsymbol{x}\in S$, $T_{\boldsymbol{x}}S$ is contained in $\boldsymbol{C}_{\varepsilon}^{\mathrm{uc}}(\boldsymbol{x})$.
Furthermore such a surface $S$ is called \emph{proper} 
if the restriction $\mathrm{proj}|_S:S\longrightarrow \mathbb{R}^2$ is a proper map onto $\mathbb{R}^2$, that is, 
$(\mathrm{proj}|_S)^{-1}(K)$ is compact for any compact subset $K$ of $\mathbb{R}^2$.
This condition is necessary to avoid the situation where the image $\widehat A$ is multi-folded in $\mathbb{R}^2$ for a subset $A$ of $S$.
In fact, since $\mathrm{proj}|_S$ is a local homeomorphism for any $\varepsilon$-almost horizontal proper surface $S$,  it follows from the standard fact of 
covering theory that $\mathrm{proj}|_S:S\longrightarrow \mathbb{R}^2$ is a surjective homeomorphism if $S$ is connected.
The $\varepsilon$-almost horizontality of $S$ will be used to show the second inequality of \eqref{eqn_Leb_A} below.

Fix connected and mutually disjoint open neighborhoods $\mathbb{U}_i$ of 
$\mathbb{X}_i$ $(i=0,1,2)$ in $(-\varepsilon_0,1+\varepsilon_0)\times (-2\varepsilon_0,1+2\varepsilon_0)^2$.
We may assume that, for any $g\in \mathcal{U}_0$, $\mathbb{U}_i$ contains $\mathbb{X}_{g,i}$ and 
$\{\mathbb{U}_0,\mathbb{U}_1\}$ is a coding pair for $\Lambda_g$.
Moreover one can suppose that $g(\mathbb{U}_i)$ $(i=0,1)$ is contained in $\mathbb{B}$.
There exist diffeomorphisms $\varphi_{g,i}:\mathbb{R}^3\longrightarrow \mathbb{R}^3$ such that 
$\varphi_{g,i}|_{\mathbb{X}_{g,i}}=g|_{\mathbb{X}_{g,i}}$ for $i=0,1$, 
$\varphi_{g,2}|_{\mathbb{X}_{g,2}}=g^2|_{\mathbb{X}_{g,2}}$ 
and 
$\varphi_{g,j}|_{\mathbb{R}^3\setminus \mathbb{U}_j}=\varphi_j|_{\mathbb{R}^3\setminus \mathbb{U}_j}$ for $j=0,1,2$.
By taking $\mathcal{U}_0$ sufficiently small, we may assume that $\varphi_{g,i}$ $(i=0,1)$ is arbitrarily $C^r$-close to $\varphi_i$ and $\varphi_{g,2}$ is arbitrarily $C^1$-close to $\varphi_2$.

\begin{lemma}\label{l_ea_horizontal}
There exists a neighborhood $\mathcal{U}_0$ of $f_0$ such that, 
for any connected $\varepsilon$-almost horizontal proper surface $S$ in $\mathbb{R}^3$, $g\in\mathcal{U}_0$ and $i=0,1,2$, 
the surface 
$\varphi_{g,i}(S)$ is also $\varepsilon$-almost horizontal and proper.
\end{lemma}
\begin{proof}
Fix \(\varepsilon>0\) sufficiently small. 
After possibly shrinking the \(C^r\)-neighborhood \(\mathcal U_0\) of \(f_0\), we may assume that, for every \(g\in\mathcal U_0\) and each relevant branch \(i\),
\begin{equation}\label{eqn_DphiDphi}
\|D\varphi_{g,i}-D\varphi_i\|\le C\varepsilon^2
\end{equation}
on the corresponding domain, where \(C>0\) is independent of \(g\). 
In particular, the coefficients \(a_{kl}\) appearing below are not regarded as functions of \(\varepsilon\); rather, once \(\varepsilon\) is fixed, the neighborhood \(\mathcal U_0\) is chosen sufficiently small so that the derivative error is of order \(\varepsilon^2\).

First we consider the case of $i=0,1$.
By \eqref{eqn_f0V} and \eqref{eqn_DphiDphi}, we have 
\begin{equation}\label{eqn_Dphi_i}
D\varphi_{g,i}(\boldsymbol{x})=
\begin{pmatrix} (-1)^i\lambda_{\rm u}+\alpha_{11}(\boldsymbol{x})& \alpha_{12}(\boldsymbol{x})& \alpha_{13}(\boldsymbol{x})\\
\alpha_{21}(\boldsymbol{x})&(-1)^i \lambda_{\rm ss}+\alpha_{22}(\boldsymbol{x})& \alpha_{23}(\boldsymbol{x})\\
\alpha_{31}(\boldsymbol{x})& \alpha_{32}(\boldsymbol{x}) &\lambda_{{\rm cs}i}+\alpha_{33}(\boldsymbol{x})
\end{pmatrix}
\end{equation}
for any $g\in\mathcal{U}_0$ and $\boldsymbol{x}=(x,y,z)\in \mathbb{R}^3$,  
where each $\alpha_{kl}:\mathbb{R}^3\longrightarrow \mathbb{R}$ is a continuous function supported on $\mathbb{U}_i$ and satisfying  
$\sup\bigl\{|\alpha_{kl}(\boldsymbol{x})|\, ;\, \boldsymbol{x} \in \mathbb{R}^3\bigr\}=O(\varepsilon^2)$.
Here we denote by $O(\varepsilon^2)$ a non-negative constant with $\limsup\limits_{\varepsilon\to +0}
O(\varepsilon^2)/\varepsilon^2<\infty$.

The definition of $\boldsymbol{C}_{\varepsilon}^{\mathrm{uc}}(\boldsymbol{x})$ 
implies $|v^{\mathrm{s}}|\leq \varepsilon$ for any $\boldsymbol{v}=(v^{\mathrm{u}},v^{\mathrm{s}},v^{\mathrm{cs}})\in 
\boldsymbol{C}_{\varepsilon}^{\mathrm{uc}}(\boldsymbol{x})$ with $(v^{\mathrm{u}})^2+(v^{\mathrm{cs}})^2=1$.
Since \(D\varphi_{g,i}(\boldsymbol{x})\) is linear and the cone condition is homogeneous, it is enough to prove the assertion for vectors normalized by
$(v^{\mathrm{u}})^2+(v^{\mathrm{cs}})^2=1$.
We set $D\varphi_{g,i}(\boldsymbol{x})(v^{\rm u},v^{\rm s},v^{\rm cs})^T
=(\widehat{v}^{\rm u},\widehat{v}^{\rm s},\widehat{v}^{\rm cs})^T$, 
where $(v_1,v_2,v_3)^T$ denotes the transposed column vector of a row vector $(v_1,v_2,v_3)$.
By \eqref{eqn_Dphi_i}, 
\[
|\widehat{v}^{\rm s}|\leq \lambda_{\rm ss}|v^{\rm s}|+O(\varepsilon^2)
\leq \varepsilon\lambda_{\rm ss}+O(\varepsilon^2).
\]
Moreover, 
\begin{align*}
\varepsilon\sqrt{(\widehat{v}^{\rm u})^2+(\widehat{v}^{\rm cs})^2}
&\geq 
\varepsilon\sqrt{\lambda_{\rm u}^2(v^{\rm u})^2+\lambda_{{\rm cs}i}^2(v^{\rm cs})^2}-O(\varepsilon^2)\\
&>\varepsilon\lambda_{{\rm cs}i}\sqrt{(v^{\rm u})^2+(v^{\rm cs})^2}-O(\varepsilon^2)\\
&
=\varepsilon\lambda_{{\rm cs}i}-O(\varepsilon^2).
\end{align*}
By \eqref{eqn_eigen_v},   
$|\widehat{v}^{\rm s}|\leq \varepsilon\sqrt{(\widehat{v}^{\rm u})^2+(\widehat{v}^{\rm cs})^2}$ holds 
if we take $\varepsilon$ sufficiently small.
It follows that $D\varphi_{g,i}(\boldsymbol{x})\boldsymbol{v}\in \boldsymbol{C}_{\varepsilon}^{\mathrm{uc}}(\boldsymbol{x})$.

In the case of $i=2$, by \eqref{eqn_phi_2}, we have 
\begin{equation}\label{eqn_Dphi2}
D\varphi_{g,2}(\boldsymbol{x})=\begin{pmatrix} \dfrac{d\eta}{dx}(x)+\beta_{11}(\boldsymbol{x})&\beta_{12}(\boldsymbol{x})&a_2+\beta_{13}(\boldsymbol{x})\\
\beta_{21}(\boldsymbol{x})&a_3+\beta_{22}(\boldsymbol{x})&\beta_{23}(\boldsymbol{x})\\
a_4+\beta_{31}(\boldsymbol{x})&\beta_{32}(\boldsymbol{x})&\beta_{33}(\boldsymbol{x})\end{pmatrix},
\end{equation}
where each $\beta_{kl}:\mathbb{R}^3\longrightarrow \mathbb{R}$ is a continuous function supported on $\mathbb{U}_2$ and satisfying  
$\sup\bigl\{|\beta_{kl}(\boldsymbol{x})|\, ;\, \boldsymbol{x} \in \mathbb{R}^3\bigr\}=O(\varepsilon^2)$.
We set again $D\varphi_{g,2}(\boldsymbol{x})(v^{\rm u},v^{\rm s},v^{\rm cs})^T
=(\widehat{v}^{\rm u},\widehat{v}^{\rm s},\widehat{v}^{\rm cs})^T$.
As in the previous case, 
\begin{equation}\label{eqn_vs_a3}
|\widehat{v}^{\rm s}|\leq \varepsilon |a_3|+O(\varepsilon^2)
\end{equation}
for any $\boldsymbol{v}=(v^{\mathrm{u}},v^{\mathrm{s}},v^{\mathrm{cs}})\in 
\boldsymbol{C}_{\varepsilon}^{\mathrm{uc}}(\boldsymbol{x})$ with 
$(v^{\mathrm{u}})^2+(v^{\mathrm{cs}})^2=1$.

First we consider the case when $|v^{\rm u}|\geq \dfrac{a_2}{2a_1}|v^{\rm cs}|$ or equivalently 
$|v^{\rm u}|\geq \dfrac{a_2}{\sqrt{4a_1^2+a_2^2}}$.
By \eqref{eqn_min} and \eqref{eqn_Dphi2}, it holds that 
\[
|\widehat{v}^{\rm cs}|\geq |a_4v^{\rm u}|-O(\varepsilon^2)\geq \dfrac{a_2|a_4|}{\sqrt{4a_1^2+a_2^2}}-O(\varepsilon^2)>|a_3|
\]
if $\varepsilon$ is taken sufficiently small.
By \eqref{eqn_vs_a3}, we may also assume that 
\begin{equation}\label{eqn_vcs>v3}
\varepsilon |\widehat{v}^{\rm cs}|>|\widehat{v}^{\rm s}|.
\end{equation}

Next we suppose that $|v^{\rm u}|<\dfrac{a_2}{2a_1}|v^{\rm cs}|$ or equivalently 
$|v^{\rm cs}|> \dfrac{2a_1}{\sqrt{4a_1^2+a_2^2}}$.
Again by \eqref{eqn_min} and \eqref{eqn_d_eta}, \eqref{eqn_Dphi2}, 
\begin{align*}
|\widehat{v}^{\rm u}|&\geq a_2|v^{\rm cs}|-\left|\frac{d\eta}{dx}(x)v^{\rm u}\right|-O(\varepsilon^2)
\geq a_2|v^{\rm cs}|-a_2\left(\frac12+\varepsilon_0\right)|v^{\rm cs}|-O(\varepsilon^2)\\
&> a_2\left(\frac12-\varepsilon_0\right)\dfrac{2a_1}{\sqrt{4a_1^2+a_2^2}}-O(\varepsilon^2)
=\dfrac{a_1a_2}{\sqrt{4a_1^2+a_2^2}}(1-2\varepsilon_0)-O(\varepsilon^2)\\
&>|a_3|-O(\varepsilon^2).
\end{align*}
Thus we have 
\begin{equation}\label{eqn_vu>v3}
\varepsilon |\widehat{v}^{\rm u}|>|\widehat{v}^{\rm s}|
\end{equation}
for any sufficiently small $\varepsilon$. 
In either case, by \eqref{eqn_vcs>v3} or \eqref{eqn_vu>v3}, 
$|\widehat{v}^{\rm s}|\leq \varepsilon\sqrt{(\widehat{v}^{\rm u})^2+(\widehat{v}^{\rm cs})^2}$.
It follows that  
$D\varphi_{g,2}(\boldsymbol{x})\boldsymbol{v}\in \boldsymbol{C}_{\varepsilon}^{\mathrm{uc}}(\boldsymbol{x})$.

Since $\varphi_{g,i}=\varphi_i$ on $\mathbb{R}^3\setminus \mathbb{U}_i$ for $i=0,1,2$, 
the restriction 
$\mathrm{proj}|_{\varphi_{g,i}(S)}:\varphi_{g,i}(S)\longrightarrow \mathbb{R}^2$ is proper and surjective.
Hence $\varphi_{g,i}(S)$ is an $\varepsilon$-almost horizontal proper surface.
This completes the proof.
\end{proof}

\section{Proof of Theorem \ref{thm_A}}\label{S_proof_of_thmA}

A subset $A$ of $\mathbb{R}^3$ is $\varepsilon$-\emph{almost horizontal} if $A$ is contained in 
an $\varepsilon$-almost horizontal proper surface $S$ in $\mathbb{R}^3$.
Since such an \(A\) is an \(\varepsilon\)-almost horizontal graph over \(\widehat A\), 
\eqref{eqn_Leb_A} follows directly from the area formula and the bound on the slope of \(A\).
\begin{equation}\label{eqn_Leb_A}
\mathrm{Leb}_S(A)\geq \mathrm{Leb}\,(\widehat A)\geq \frac{\mathrm{Leb}_S(A)}{\sqrt{1+\varepsilon^2}}
=\mathrm{Leb}_S(A)(1-O(\varepsilon^2)).
\end{equation}
Here $\mathrm{Leb}_S(\cdot)$ denotes the Lebesgue measure on $S$ with respect to 
the Riemannian metric on $S$ induced from the Euclidean metric on $\mathbb{R}^3$ 
and $\mathrm{Leb}\,(\cdot)$ denotes the standard Lebesgue measure on $\mathbb{R}^2$.

Recall that, for any $g\in \mathcal{U}_0$, the continuation $\Lambda_g$ of $\Lambda_{f_0}$ is associated with the coding pair $\{ \mathbb{U}_0, \mathbb{U}_1 \}$ given in Subsection \ref{ss_almost_h}.
Let $\mathcal{I}_g: \Lambda_{g}  \longrightarrow  \{ 0,1\} ^{\mathbb Z}$ be the coding map determined from the pair. 
We say that an element $\boldsymbol{x}_g$ of $\Lambda_g$ with $\mathcal{I}_g(\boldsymbol{x}_g)=(v_j)_{j\in \mathbb{Z}}$ satisfies the \emph{asymptotically 1-filling condition} if 
\begin{equation}\label{eqn_1-filling}
\lim_{n\to \infty}\frac{\#\left\{j\,;\,0\leq j\leq n-1,\ v_j=1\right\}}{n}=1.
\end{equation}
In other words, the symbol $1$ occupies asymptotically almost all
time indices; equivalently, the set of indices carrying the symbol
$0$ has asymptotic density zero. This is a global density condition
on the whole initial block $\{0,\ldots,n-1\}$, in contrast with the
majority condition \eqref{eqn_mj_cond}, which concerns moving terminal blocks of
length of order $n^{2/3}$.

We note that the subset of \(\Lambda_g\) consisting of elements with the asymptotically \(1\)-filling condition is dense in \(\Lambda_g\). 
Indeed, any finite binary word can be extended to an asymptotically \(1\)-filling itinerary simply by taking all sufficiently large symbols to be \(1\).

This dense family should be distinguished from the majority-condition family 
\(\Lambda_g^{(\mathrm{mj})}\) recalled in \eqref{eqn_mj_cond}. The asymptotically \(1\)-filling condition is a global density condition requiring the frequency of the symbol \(1\) to tend to one, whereas the majority condition requires a predominance of the symbol \(0\) in the terminal windows appearing in \eqref{eqn_mj_cond}. The latter family is also dense, since any finite binary word can be extended by taking all sufficiently large symbols to be \(0\).

\begin{remark}\label{r_ab+2}
As in the proof of \cite[Lemma 2.2]{KNS}, we may assume that $\alpha_k$ and $\beta_k$ in Definition \ref{describable} satisfy  
\begin{equation}\label{eqn_ab+2}
\alpha_k+\beta_k+2\leq \alpha_{k+1}
\end{equation}
for any $k\geq 1$.
In the case when $\alpha_k+\beta_k+1=\alpha_{k+1}$, we consider the new interval 
$[\alpha_k,\alpha_{k+1}+\beta_{k+1}]\cap \mathbb{Z}$ instead of $\mathbb{I}_k\cup \mathbb{I}_{k+1}$. 
Repeating the construction, we have a sequence of intervals, still denoted by $(\mathbb{I}_k)$,  which satisfies not only \eqref{D1}, \eqref{D2} but also \eqref{eqn_ab+2}.
When $(\mathbb{I}_k)$ consists of only finitely many intervals, 
the last entry $\mathbb{I}_{k_0}$ is a half open interval $[\alpha_{k_0},\infty)\cap 
\mathbb{Z}$.
Then we split $\mathbb{I}_{k_0}$ into infinitely many intervals such that
$\mathbb{I}_{k_0}^{\mathrm{new}}=[\alpha_{k_0},\alpha_{k_0}+2]\cap \mathbb{Z}$ and 
$\mathbb{I}_{k_0+i}^{\mathrm{new}}=[\alpha_{k_0}+2^{i+1},\alpha_{k_0}
+2^{i+2}-2]\cap\mathbb{Z}$ for $i\geq 1$, 
which still satisfies \eqref{D1}, \eqref{D2} and \eqref{eqn_ab+2}.
See Figure \ref{f_remark61}. 
\end{remark}
\begin{figure}[hbtp]
\centering
\scalebox{0.6}{\includegraphics[clip]{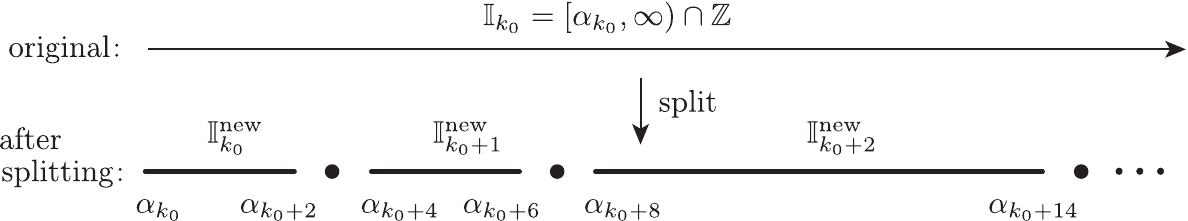}}
\caption{Schematic illustration of the splitting in Remark \ref{r_ab+2}.
The half-infinite interval $I_{k_0}$ is divided into intervals
$\mathbb{I}^{\rm new}_{k_0+i}$ separated by a single time index. 
The omitted indices have asymptotic density zero, while the separation ensures
condition \eqref{eqn_ab+2}.
}
\label{f_remark61}
\end{figure}

The proof of Theorem~\ref{thm_A} is based on a comparison between expansion and loss along prescribed symbolic itineraries. Iterates contained in the expanding region \(\mathbb{X}_1\) produce uniform projected-area growth. 
There are, however, two possible sources of loss: the symbol-\(0\) times inside the prescribed intervals, which contribute the factor \(\rho_0\), possibly smaller than one, and the transition or branching times, including passages through \(\mathbb{X}_2\), outside the prescribed intervals. 
For asymptotically \(1\)-filling itineraries, both types of exceptional times have asymptotic density zero, while the frequency of visits to \(\mathbb{X}_1\) tends to one. 
Hence the expansion accumulated during the symbol-\(1\) times dominates all possible losses.

Using describability obtained from the Pluripotency Lemma, we construct a sequence \((A_{\alpha_k})_{k\geq 2}\) of almost horizontal sections whose projected areas therefore become arbitrarily large. 
Since all such sections remain inside the bounded cube \(\mathbb{B}\), this yields the desired contradiction.

\begin{proof}[Proof of Theorem \ref{thm_A}]
Suppose that there existed an element $g_*$ of $\mathcal{U}_0$ which is strongly 
pluripotent for the continuation $\Lambda_{g_*}$ of the whole $\Lambda_{f_0}$.
Take an element $\boldsymbol{x}_*\in \Lambda_{g_*}$ with asymptotically 1-filling condition.
We set $\mathcal{I}_{g_*}(\boldsymbol{x}_*)=(v_j)_{j\in \mathbb{Z}}$ and $\Sigma=\{v\}$ 
for $v=(v_j)_{j\in \mathbb{N}_0}$.
Since $g_*$ is strongly pluripotent for the subset $\mathcal{I}_{g_*}^{-1}(\widehat\Sigma)$ of the whole $\Lambda_{g_*}$, 
$g_*$ is $\Sigma$-describable by Theorem \ref{p-lemma}, that is,  
there exist $g\in \mathcal{U}_0$ arbitrarily $C^r$-close to $g_*$,  
a subset $D_g$ of $M$ with positive Lebesgue measure and 
non-negative integers $\alpha_k$, $\beta_k$ satisfying \eqref{D1}, \eqref{D2} 
and \eqref{eqn_ab+2}.

By using \eqref{D1}, one can show that $\beta_{k}\neq 0$ for infinitely many $k\in \mathbb{N}$.
In fact, if $\beta_{k}=0$ for all but finitely many $k\in \mathbb{N}$, then by \eqref{eqn_ab+2} 
we would have 
\[
\limsup _{n\rightarrow \infty} \frac{\#\left\{j\,;\, 0\le j \le n-1,\ j \in \bigcup _{k=1}^\infty \mathbb{I}_k\right\}}{n}\leq \frac12,
\]
which contradicts \eqref{D1}.
In particular we may assume that $\beta_1\geq 1$ by regarding $k-k_0+1$ as a new $k$ for some $k_0$ with 
$\beta_{k_0}\neq 0$.
Since $g^{\alpha_1}(D_g)\subset \mathbb{U}_a$ and $g^{\alpha_1+1}(D_g)\subset \mathbb{U}_b$ for some $a,b\in \{0,1\}$, 
$g^{\alpha_1+1}(D_g)$ is contained in $g(\mathbb{U}_a)\cap \mathbb{U}_b\subset \mathbb{B}$, 
where $\mathbb{U}_i$ $(i=0,1)$ is the open neighborhood of $\mathbb{X}_{g,i}$ in $M$ given in Section \ref{S_near_f0}.
Let $\mathbb{T}_g$ be the union $\mathbb{X}_{g,0}\cup \mathbb{X}_{g,1}\cup \mathbb{X}_{g,2}\cup g(\mathbb{X}_{g,2})$. 
We first show that, for every \(j\geq \alpha_1+1\), one has \(g^j(D_g)\subset \mathbb T_g\).
If not, there would exist $j\geq \alpha_1+1$ with $g^j(D_g)\setminus \mathbb{T}_g\neq \emptyset$.
We denote by $j_0$ the smallest integer among such $j$. 
Then $g^{j_0}(D_g)$ contains an element $\boldsymbol{y}$ of $\mathbb{Y}_g\cup \mathbb{S}_g\setminus 
\mathrm{Fr}(\mathbb{Y}_g\cup \mathbb{S}_g)$, 
which is obvious when $j_0=\alpha_1+1$ and proved by using the fact $g^{j_0-1}(D_g)\subset \mathbb{T}_g$ 
when $j_0>\alpha_1+1$.
By Extra Assumptions with respect to $g$, $(g^{n}(\boldsymbol{y}))_{n\geq 0}$ converges to the attracting fixed point $s_{g,0}$ of $g$ which is the continuation of $s_0$.
Since \(\alpha_k-j_0\to\infty\), for all sufficiently large \(k\) the point \(g^{\alpha_k-j_0}(\boldsymbol{y})\) belongs to a small neighborhood of \(s_{g,0}\) disjoint from \(\mathbb{U}_0\cup\mathbb{U}_1\). 
However, since \(\alpha_k\in I_k\), condition \eqref{D2} gives
$$
g^{\alpha_k-j_0}(\boldsymbol{y})\in g^{\alpha_k}(D_g)
\subset \mathbb{U}_{v_{\alpha_k}}
\subset\mathbb{U}_0\cup\mathbb{U}_1,
$$
a contradiction.
Thus we have $g^j(D_g)\subset \mathbb{T}_g$ for any $j\geq \alpha_1+1$.
This implies that $g^j(D_g)$ lies in either $\mathbb{X}_{g,0}$ or $\mathbb{X}_{g,1}$ for any $j\in \mathbb{I}_k$.

By Fubini's Theorem, there exists $y_0\in (-\varepsilon_0,1+\varepsilon_0)$ such that 
the horizontal section $A_{\alpha_2}=g^{\alpha_2}(D_g)\cap \{(x,y,z)\in\mathbb{B}\,;\, y=y_0\}$ has positive 2-dimensional Lebesgue measure.

We set
\[\gamma_k=\alpha_{k+1}-(\alpha_k+\beta_k)\]
and define inductively subsets $A_j$ $(j> \alpha_2)$ of $M$ each of which is contained in a component of $\mathbb{T}_g$.

First suppose that $A_j$ $(\alpha_2\leq j\leq \alpha_k)$ are already defined for some $k\geq 2$.
\begin{enumerate}[(1)]
\item
Let $A_j=g^{j-\alpha_k}(A_{\alpha_k})$ for $j=\alpha_k+1,\dots,\alpha_k+\beta_k$, 
which are contained in either $\mathbb{X}_{g,0}$ or $\mathbb{X}_{g,1}$.
\item
If $A_{\alpha_k+\beta_k+j-1}$ is contained in either $\mathbb{X}_{g,0}$ or $\mathbb{X}_{g,1}$  
for some $j$ with $1\leq j\leq \gamma_k$, then 
$g(A_{\alpha_k+\beta_k+j-1})$ is a disjoint  union of $C_{\alpha_k+\beta_k+j;i}=
g(A_{\alpha_k+\beta_k+j-1})\cap \mathbb{X}_{g,i}$ $(i=0,1,2)$.
Possibly some of them are empty.
We suppose that $A_{\alpha_k+\beta_k+j}$ is one of $C_{\alpha_k+\beta_k+j;i}$ with 
$\mathrm{Leb}\,(\widehat C_{\alpha_k+\beta_k+j;i})\geq \mathrm{Leb}\,(\widehat C_{\alpha_k+\beta_k+j;s})
$ 
for any $s\in \{0,1,2\}\setminus \{i\}$.
In particular, 
\vspace{-8pt}

\begin{figure}[H]
\centering
\scalebox{0.68}{\includegraphics[clip]{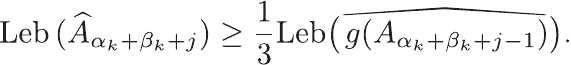}}
\end{figure}
\vspace{-8pt}
\item
If 
$A_{\alpha_k+\beta_k+j-1}$ is contained in $\mathbb{X}_{g,2}$ for some $j$ with $1\leq j\leq  \gamma_k-1$, then we set 
$A_{\alpha_k+\beta_k+j}=g(A_{\alpha_k+\beta_k+j-1})$ and $A_{\alpha_k+\beta_k+j+1}=g^2(A_{\alpha_k+\beta_k+j-1})$.
\end{enumerate}

Note that 
$
\widehat C_{\alpha_k+\beta_k+j;i}
\subset \widehat{\mathbb X}_{g,i}
$
and the sets \(\widehat{\mathbb X}_{g,0}\), \(\widehat{\mathbb X}_{g,1}\), and \(\widehat{\mathbb X}_{g,2}\) are pairwise disjoint.
Since moreover $\widehat C_{\alpha_k+\beta_k+j;i}$ $(i=0,1,2)$ are contained in 
the same connected $\varepsilon$-almost horizontal surface, 
they are pairwise disjoint, 
Hence there is no overlap issue in comparing their projected areas.

By Lemma \ref{l_ea_horizontal}, any $A_j$ in $\mathbb{T}_g\setminus g(\mathbb{X}_{g,2})$ is $\varepsilon$-almost horizontal.
Take $\varepsilon>0$ sufficiently small.
Then, by \eqref{eqn_Leb_P0}, \eqref{eqn_Leb_P} and \eqref{eqn_Leb_A}, there exist positive constants 
\[
\rho_i=\lambda_{{\rm cs}i}\lambda_{\rm u}-O(\varepsilon^2)\quad (i=0,1),\quad 
\rho_2=|a_2a_4|-O(\varepsilon^2)
\]
such that 
\begin{equation}\label{eqn_LebA}
\begin{split}
\mathrm{Leb}\,(\widehat A_{j+1})&\geq \rho_i\mathrm{Leb}\,(\widehat A_{j})\\
&\quad\text{if $\alpha_k\leq j\leq \alpha_k+\beta_k-1$ and $A_{j}\subset \mathbb{X}_{g,i}$ $(i=0,1)$},\\
\mathrm{Leb}\,(\widehat A_{j+1})&\geq \frac13\rho_i\mathrm{Leb}\,(\widehat A_{j})\\
&\quad\text{if $\alpha_k+\beta_k\leq j\leq \alpha_{k+1}-1$ and $A_{j}\subset \mathbb{X}_{g,i}$ $(i=0,1)$},\\
\mathrm{Leb}\,(\widehat A_{j+2})&\geq \rho_2\mathrm{Leb}\,(\widehat A_{j})\\
&\quad\text{if $\alpha_k+\beta_k+1\leq j\leq \alpha_{k+1}-2$ and $A_{j}\subset \mathbb{X}_{g,2}$.}
\end{split}
\end{equation}

Here we note that $A_{\alpha_{k+1}-1}$ is not contained in $\mathbb{X}_{g,2}$.
If not, by Extra Assumptions $A_{\alpha_{k+1}}\subset g(\mathbb{X}_{g,2})$ would be disjoint from $\mathbb{X}_{g,0}\cup \mathbb{X}_{g,1}$, a contradiction.
By \eqref{eqn_lamcs1}, one can suppose that $\rho_1>1$.
Now we consider the constant $\widehat\rho=\min\left\{\dfrac{\rho_0}3,\dfrac{\rho_1}3,\rho_2,1\right\}$.
By Definition \ref{describable}\,\eqref{D1}, $\lim\limits_{k\to \infty}\sum_{i=2}^k \gamma_i/\sum_{i=2}^k\beta_i=0$.
Let $\tau_k=\#\{j\in \mathbb{I}_k\,;\,v_j=0\}$.
Then the asymptotically 1-filling condition \eqref{eqn_1-filling} implies $\lim\limits_{k\to \infty}\sum_{i=2}^k \tau_i/\sum_{i=2}^k\beta_i=0$ 
and hence in particular $\frac12\sum_{i=2}^k \beta_i\geq \sum_{i=2}^k\tau_i$ for all sufficiently large $k$.

We now iterate the estimates in (6.4). For each \(i\), among the \(\beta_i\) times in the prescribed interval \(I_i\), exactly \(\beta_i-\tau_i\) correspond to the symbol \(1\), and hence contribute the expanding factor \(\rho_1\). 
The remaining \(\tau_i\) symbol-\(0\) times contribute factors bounded from below by \(\widehat\rho\).
Let $\gamma_k^{(3)}$ denote the number of integers $j$ with $1 \le j \leq \gamma_k - 1$ for which (3) holds.
During the transition interval of length \(\gamma_i\), all factors are also bounded from below by \(\widehat\rho\), except that in each case of item~(3) two iterates are treated at once; this accounts for the subtraction of \(\gamma_i^{(3)}\). 
Therefore, multiplying (6.4) along the construction from \(A_{\alpha_2}\) to \(A_{\alpha_{k+1}}\), we obtain
$$
\operatorname{Leb}(\widehat A_{\alpha_{k+1}})
\ge
\rho_1^{\sum_{i=2}^k(\beta_i-\tau_i)}
\widehat\rho^{\sum_{i=2}^k(\gamma_i-\gamma_i^{(3)}+\tau_i)}
\operatorname{Leb}(\widehat A_{\alpha_2}).
$$
Since \(\widehat\rho\le1\), \(\gamma_i^{(3)}\ge0\) and
$$
\sum_{i=2}^k\tau_i\le\frac12\sum_{i=2}^k\beta_i
$$
for all sufficiently large \(k\), the preceding estimate yields
$$
\operatorname{Leb}(\widehat A_{\alpha_{k+1}})
\ge
\rho_1^{\frac12\sum_{i=2}^k\beta_i}
\widehat\rho^{\sum_{i=2}^k(\gamma_i+\tau_i)}
\operatorname{Leb}(\widehat A_{\alpha_2}).
$$
Taking logarithms gives
\begin{align*}
\log \mathrm{Leb}\,(\widehat A_{\alpha_{k+1}})&\geq 
\frac12\sum_{i=2}^k\beta_i\log \rho_1+
\sum_{i=2}^k(\gamma_i+\tau_i)\log \widehat\rho+\log\mathrm{Leb}\,(A_{\alpha_2})\\
&=\frac12\sum_{i=2}^k\beta_i\left(\log \rho_1+\frac{2\sum_{i=2}^k(\gamma_i+\tau_i)}{\sum_{i=2}^k\beta_i}\log \widehat\rho\right)+\log\mathrm{Leb}\,(A_{\alpha_2})\to \infty
\end{align*}
as $k\to \infty$ and so $\lim\limits_{k\to \infty}\mathrm{Leb}\,(\widehat A_{\alpha_{k+1}})=\infty$, which contradicts that $\widehat A_{\alpha_{k+1}}$ is a subset of $I_{\varepsilon_0}^2$.
This completes the proof.
\end{proof}

\begin{proof}[Proof of Corollary \ref{cor_B}]
Fix \(2\leq r<\infty\). We choose \(f_0\) to be the affine model constructed in \cite[Subsection~3.1]{KNS}, with the parameters satisfying
\begin{equation}\label{eqnKNS}
\lambda_{\rm ss}<\lambda_{\rm cs0}<\frac12<\lambda_{\rm cs1}<1<
\lambda_{\rm cs0}+\lambda_{\rm cs1},\quad
2<\lambda_{\rm u},\quad
\lambda_{\rm cs0}\lambda_{\rm cs1}\lambda_{\rm u}^{2}<1.
\end{equation}
These are precisely conditions (3.1a) and (3.1b) of \cite{KNS}. 
The geometric roles of these parameter conditions in the affine blender construction are discussed in \cite[Section~3]{KNS}. 
In the present proof, we only use the fact that \eqref{eqnKNS} implies the parameter assumptions \eqref{eqn_eigen_v} of Theorem~\ref{thm_A}.

The additional assumptions imposed in Section \ref{S_2} can be arranged without changing the local construction of \(f_0\) in \cite{KNS}. 
Indeed, as explained immediately after the Extra Assumptions, the global extension of \(f_0\) may be chosen so as to satisfy those assumptions while agreeing with the affine model in \cite{KNS} on the region relevant to the blender-horseshoe construction. 
Thus \(f_0\) may be chosen simultaneously to satisfy the hypotheses used in \cite[Theorem~B]{KNS} and all the assumptions required for Theorem~\ref{thm_A}.

Let \(\widehat{\mathcal O}\) be a sufficiently small \(C^r\)-neighborhood of \(f_0\) for which \cite[Theorem~B]{KNS} applies. By shrinking the neighborhood \(\mathcal U_0\) given by Theorem~\ref{thm_A}, if necessary, we may suppose that
$$
\mathcal U_0\subset  \widehat{\mathcal O}.
$$
Then every \(f\in\mathcal U_0\) is strongly pluripotent for
\(\Lambda_f^{({\rm mj})}\) by \cite[Theorem~B]{KNS}, whereas Theorem~\ref{thm_A} implies that \(f\) is not strongly pluripotent for the whole blender-horseshoe \(\Lambda_f\). 
This proves the corollary.
\end{proof}

\section{Robust strong pluripotency for vector fields}\label{S_flow}

In this section we explain how the dichotomy in Corollary \ref{cor_B} can be realized by
first-return maps of an open family of vector fields.  
We also record a local realization lemma showing that the perturbing diffeomorphisms 
occurring in the definition of strong pluripotency may be realized as first-return maps 
of small perturbations of the vector field.  
Thus the continuous-time realization below is robust not only at the level of the reference return map, but also at the
level of the perturbations used to produce observable sets.

Throughout this section, let $2\le r<\infty$, and let $f_0\in\mathrm{Diff}^r(M)$ and
$\mathcal U_0$ be as in Corollary \ref{cor_B}.  Since the local affine model and the
global extensions used in the construction of $f_0$ can be chosen smooth, we may 
choose $f_0$ of class $C^\infty$ and regard $\mathcal U_0$ as a $C^r$-neighborhood of $f_0$. 
We suppose that $\Lambda_f$ is the continuation of the
affine blender-horseshoe and $\Lambda_f^{(\mathrm{mj})}$ is the subset
satisfying the majority condition.  After shrinking $\mathcal U_0$ if
necessary, we may assume that the relevant wild blender-horseshoe structure
persists throughout $\mathcal U_0$.

\subsection{The mapping torus and its global section}

Consider the mapping torus
\[
 N_{f_0}
 =\frac{M\times[0,1]}{(x,1)\sim(f_0(x),0)}.
\]
Since $f_0$ is smooth, $N_{f_0}$ is a smooth closed manifold of dimension
$\dim M+1$.  Let $X_0$ be the vector field induced by $\partial/\partial t$
on $M\times[0,1]$.  Its flow is denoted by $\Phi_{X_0}^t$.  The submanifold
\[
 \Sigma=\{[x,0]\,;\,x\in M\}\subset N_{f_0}
\]
is a global cross-section of $N_{f_0}$.  
Identifying $\Sigma$ with $M$, its first-return map is exactly
\[
 P_{X_0}=f_0,
 \]
and the return time is identically equal to $1$.

There is a natural smooth fibration
\[
 \vartheta:N_{f_0}\longrightarrow\mathbb{S}^1,
 \quad \vartheta([x,t])=t\pmod 1.
\]
It satisfies $d\vartheta(X_0)=1$.  Hence, if $X$ is sufficiently $C^r$-close
to $X_0$, then
\begin{equation}\label{eq:positive-theta}
 d\vartheta(X)>0
\end{equation}
everywhere.  In particular, $\Sigma=\vartheta^{-1}(0)$ is still a global
cross-section for $X$.  We denote its first-return map and return-time function
by
\[
 P_X:\Sigma\longrightarrow\Sigma,
 \quad
 \tau_X:\Sigma\longrightarrow(0,\infty),
\]
respectively.  By the continuous dependence of flows and return times on the
vector field, $P_X$ is $C^r$-close to $f_0$.

First we present a standard realization fact in a form suitable to our arguments.
Here ${\mathfrak X}^r(N_{f_0})$ denotes the space of $C^r$ vector fields on $N_{f_0}$
with the $C^r$ topology.

\begin{lemma}[Local realization of Poincar\'e maps]\label{l_local_realization}
For any $X\in\mathfrak X^r(N_{f_0})$ sufficiently $C^r$-close to $X_0$ and 
any \(C^r\)-neighborhood \(\mathcal V\) of \(X\) in \(\mathfrak X^r(N_{f_0})\), there exists a \(C^r\)-neighborhood \(\mathcal W\) of \(P_X\) in \(\operatorname{Diff}^r(\Sigma)\) such that, for every \(g\in\mathcal W\), there exists \(Y\in\mathcal V\) satisfying
\begin{equation}\label{eqn_PYg}
P_Y=g.
\end{equation}
Moreover \(Y-X\) is supported on an arbitrarily small flow box adjacent to \(\Sigma\).
\end{lemma}

\begin{proof}
Since $\Sigma$
is compact, $\tau_*:=\min_{x\in\Sigma}\tau_X(x)>0$.  
Fix $0<\varepsilon<\tau_*/3$.  
Then the map
\[
 \Psi:\Sigma\times[0,\varepsilon]\longrightarrow N_{f_0},
 \quad
 \Psi(x,t)=\Phi_X^t(x),
\]
defines an embedded flow box in $N_{f_0}$.

Let $g$ be an element of $\\mathrm{Diff}^r(\Sigma)$ sufficiently $C^r$-close to $P_X$ and put
\begin{equation}\label{eq:def-h}
 h=P_X^{-1}\circ g.
\end{equation}
Then $h$ is $C^r$-close to the identity.  
So there is a $C^r$ isotopy $(h_s)_{0\le s\le1}$ from
$h_0=\mathrm{id}_\Sigma$ to $h_1=h$, all of whose elements are $C^r$-close to
the identity.  For instance, one may use a Riemannian exponential map on
$\Sigma$ and, for $h$ sufficiently close to the identity, set
\[
 h_s(x)=\exp_x\bigl(s\exp_x^{-1}(h(x))\bigr).
\]
Define the time-dependent vector field $Z_s$ on $\Sigma$ by
\begin{equation}\label{eq:Zs}
 \frac{d}{ds}h_s=Z_s\circ h_s.
\end{equation}
The family $Z_s$ converges to zero in the $C^r$ topology as $h$ converges to the
identity.

Choose a non-decreasing $C^{r+1}$-function
$\eta:[0,\varepsilon]\longrightarrow [0,1]$ such that $\eta\equiv 0$ near $0$ and $\eta\equiv 1$ near
$\varepsilon$.  In the flow-box coordinates, the flow $\widetilde Y$ is defined by 
\[
\widetilde Y(x,t)
 =\partial_t+\eta'(t)Z_{\eta(t)}(x),
 \]
where $\partial_t=X$.
Since $\eta'$ vanishes near both boundary components of the flow box,
$\widetilde Y$ agrees with $X=\partial_t$ there.  We may therefore define a
global vector field $Y$ by setting $Y=\Psi_*\widetilde Y$ on the flow box and
$Y=X$ outside it.  
See Figure \ref{f_flowY}. 
If $g$ is sufficiently $C^r$-close to $P_X$, then $Y$ belongs to
$\mathcal V$.
\begin{figure}[hbtp]
\centering
\scalebox{0.6}{\includegraphics[clip]{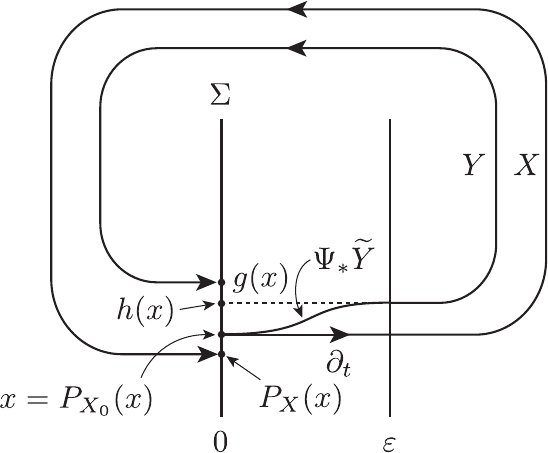}}
\caption{The flows of $X$ and $Y$.}
\label{f_flowY}
\end{figure}

By \eqref{eq:Zs} and $s=\eta(t)$, 
\[
 \frac{d}{dt}h_{\eta(t)}(x)
 =\eta'(t)Z_{\eta(t)}\bigl(h_{\eta(t)}(x)\bigr).
\]
Hence the $\widetilde Y$-orbit starting at $(x,0)$ is
$ t\longmapsto\bigl(h_{\eta(t)}(x),t\bigr)$.
This orbit leaves the flow box through $(h(x),\varepsilon)$.  From that
point on the vector fields $X$ and $Y$ coincide until the next intersection
with $\Sigma$.
So we have by \eqref{eq:def-h}
\[
 P_Y(x)=P_X(h(x))=g(x).
\]
This proves
\eqref{eqn_PYg}.  Since $\varepsilon$ may be chosen arbitrarily small, the
support of $Y-X$ may be confined to an arbitrarily small flow box.
\end{proof}

\subsection{The pluripotency dichotomy for first-return maps}

The following proposition gives the promised vector-field realization of Corollary \ref{cor_B}, which is Corollary \ref{cor_C} stated in Section \ref{S_mian_thm}.

\begin{proposition}[Robust vector-field realization]\label{cor:flow-realization}
Let $2\le r<\infty$.  There exists a $C^r$-neighborhood $\mathcal V_0$ of
$X_0$ in ${\mathfrak X}^r(N_{f_0})$ such that, for every $X\in\mathcal V_0$, the
following conditions hold.
\begin{enumerate}[\rm (1)]
\item
The first-return map satisfies $P_X:\Sigma\longrightarrow \Sigma$ belongs to 
$\mathcal U_0$.
\item
Let $\Lambda_{P_{X}}$ be the continuation of the wild affine
blender-horseshoe $\Lambda_{P_{X_0}}$.
Then $P_X$ is strongly pluripotent for
$\Lambda_{P_X}^{(\mathrm{mj})}$, but it is not so for the
whole $\Lambda_{P_X}$.
\item
For every $x\in\Lambda_{P_X}^{(\mathrm{mj})}$, 
there exist a vector field $Y$ arbitrarily $C^r$-close to
$X$ and a positive Lebesgue measure subset
$D_{P_Y}$ of $\Sigma$ such that, for the continuation $x_{P_Y}$ of $x$,
\begin{equation}\label{eq:flow-strong-pluri}
 \lim_{n\to\infty}\frac1n\sum_{j=0}^{n-1}
 \sup_{y\in D_{P_Y}}
 \mathrm{dist}\bigl(P_Y^j(y),P_Y^j(x_{P_Y})\bigr)=0.
\end{equation}
\item
The obstruction to full strong pluripotency remains valid if the admissible
perturbations are restricted to first-return maps of vector fields close to
$X$.  
In particular, there is $x\in\Lambda_{P_X}$ for which no analogue of
\eqref{eq:flow-strong-pluri} can be obtained by arbitrarily small vector-field
perturbations.
\end{enumerate}
\end{proposition}

\begin{proof}
By \eqref{eq:positive-theta}, after taking a sufficiently small neighborhood
$\mathcal V_0$ of $X_0$, the section $\Sigma$ is global for every
$X\in\mathcal V_0$.  
The first-return map depends continuously on $X$ in the
$C^r$ topology.  Shrinking $\mathcal V_0$ again if necessary, we have
$ P_X\in\mathcal U_0$ for every $X\in\mathcal V_0$.
This proves (1).  
Statement (2) follows immediately from Corollary \ref{cor_B} 
after shrinking $\mathcal U_0$ so that the wild blender-horseshoe structure persists.

We prove (3).  Fix $X\in\mathcal V_0$, a point
$x\in\Lambda_{P_X}^{(\mathrm{mj})}$, and a $C^r$-neighborhood $\mathcal V$ of
$X$.  By Lemma \ref{l_local_realization}, there is a neighborhood
$\mathcal W$ of $P_X$ such that every $g\in\mathcal W$ is the first-return map
of some $Y\in\mathcal V$.  
Since \(P_X\) is strongly pluripotent for \(\Lambda_{P_X}^{(\mathrm{mj})}\), Definition \ref{d_pluri} allows the perturbing map \(g\) to be chosen arbitrarily close to \(P_X\).
In particular it may be chosen in \(\mathcal W\). 
For this \(g\), there exists a positive-Lebesgue-measure set \(D_g\subset\Sigma\) 
satisfying the defining limit condition for strong pluripotency in Definition \ref{d_pluri}.
Realizing \(g=P_Y\) by Lemma \ref{l_local_realization} and putting \(D_{P_Y}=D_g\), 
we have \eqref{eq:flow-strong-pluri}.

Finally, suppose that (4) were false. Then, for every \(x\in\Lambda_{P_X}\) and every neighborhood of \(X\), one could find a nearby vector field \(Y\) whose first-return map \(P_Y\) satisfies the condition in Definition \ref{d_pluri} for the continuation of \(x\). 
Since \(P_Y\) is, in particular, a diffeomorphism arbitrarily close to \(P_X\), this would imply that \(P_X\) is strongly pluripotent for the whole \(\Lambda_{P_X}\), contradicting (2).
Hence the full
obstruction persists under the restriction to vector-field perturbations.
\end{proof}

\subsection{The corresponding invariant set of the flow}

For completeness, we indicate explicitly how the geometric structures of the
return map appear in the flow.  For $X\in\mathcal V_0$, 
set
\[
 \widehat\Lambda_X
 =\left\{\Phi_X^t(x)\,;\, x\in\Lambda_{P_X},\ 0\le t\le\tau_X(x)\right\}.
\]
Then $\widehat\Lambda_X$ is a compact invariant hyperbolic set for the flow 
with the usual splitting
\[
 T_{\widehat\Lambda_X}N_{f_0}
 =E_X^{\mathrm{s}}\oplus\langle X\rangle\oplus E_X^{\mathrm{u}},
\]
and the first-return dynamics of $\widehat\Lambda_X$ on $\Sigma$ is precisely
$P_X|_{\Lambda_{P_X}}$.  In particular, the affine blender-horseshoe is seen
without any loss on the global section.

A homoclinic tangency of the return map also gives a homoclinic tangency for
the flow.  
Indeed, let $q\in\Sigma$ be a tangency point between the stable and
unstable manifolds associated with $\Lambda_{P_X}$.  
At the point $q$, we have 
\[
 T_qW^{\mathrm{s}}(\widehat\Lambda_X)
 =T_qW^{\mathrm{s}}(\Lambda_{P_X})\oplus\langle X(q)\rangle,
 \quad
 T_qW^{\mathrm{u}}(\widehat\Lambda_X)
 =T_qW^{\mathrm{u}}(\Lambda_{P_X})\oplus\langle X(q)\rangle.
\]
Since the corresponding tangent spaces in $T_q\Sigma$ fail to be transverse,
their flow-saturated tangent spaces fail to be transverse in $T_qN_{f_0}$.
Thus the flow possesses the homoclinic tangency corresponding to the wild
return dynamics.

We emphasize that the positive-measure set occurring in strong pluripotency
lives naturally on the cross-section, since the notion is applied to the
first-return map.  If $D\subset\Sigma$ has positive $\dim M$-dimensional
Lebesgue measure, then, for sufficiently small $\delta>0$, its short flow tube
$
 \{\Phi_X^t(y)\,;\,y\in D,\ 0<t<\delta\}
$ 
has positive $(\dim M+1)$-dimensional volume in $N_{f_0}$ by the flow-box
coordinates.

\begin{remark}
The realization above uses a global cross-section and is designed to transfer
exactly the discrete-time mechanism of this paper to nearby vector fields.  
A more geometric continuous-time problem would be to start from an independently
defined heteroclinic network and derive the required blender and tangency
structures from the geometry of the flow itself.  
In the spirit of Labouriau-Rodrigues \cite{LR17}, it is natural to ask for a
$C^r$-open family of vector fields on a four-dimensional manifold having a
heteroclinic network of periodic orbits whose three-dimensional Poincar\'e maps
contain wild affine blender-horseshoes of the type considered here.

In such a model, symbolic itineraries would represent switching paths through
the network.  
The positive part of Corollary \ref{cor_B} would then suggest strong
observable realization of switching paths satisfying the majority condition,
whereas the obstruction in Theorem \ref{thm_A} would prevent a full realization
statement for all admissible paths.  
A genuinely continuous-time formulation would also have to incorporate the non-constant return-time function, because
physical-time averages are weighted by the residence times near the components
of the network.  
We leave this intrinsic flow realization as a problem for future study.
\end{remark}

\medskip

\subsection*{Use of Generative-AI tools declaration}
OpenAI's ChatGPT was used to assist in drafting an early
version of Section \ref{S_flow}.
The final section was substantially revised and rewritten
by the authors, who take full responsibility for its content.

\subsection*{Acknowledgments}
The authors thank the anonymous referees for their careful reading and valuable comments, which significantly improved the clarity and presentation of the manuscript.
S.~Kiriki was supported by JSPS KAKENHI Grant Number 21K03332.
X.~Li was supported by NSFC Grant Numbers 11701199 and 12331005.
Y.~Nakano was supported by JSPS KAKENHI Grant  Number 23K03188 and JST PRESTO Grant Number JPMJPR25K8.
T.~Soma was supported by JSPS KAKENHI Grant Number 22K03342.

\end{document}